\documentclass[a4paper,11pt,reqno]{amsart}
\usepackage{amsmath, amssymb, amsthm, tikz, geometry}

\usetikzlibrary{matrix,positioning,decorations.pathreplacing}
\newtheorem{theorem}{Theorem}[section]
\newtheorem{lemma}[theorem]{Lemma}
\newtheorem{proposition}[theorem]{Proposition}
\newtheorem{corollary}[theorem]{Corollary}
\newcounter{claimcounter}
\newenvironment{claim}{\stepcounter{claimcounter}{\bf{Claim \theclaimcounter.}}}{}
\newtheorem*{claimnonum}{Claim}
\newenvironment{claimproof}[0]{\par\noindent\emph{Proof of claim:}}{\hfill $\blacksquare$}
\newcommand\GreenL{\mathcal{L}}
\newcommand\GreenR{\mathcal{R}}
\newcommand\GreenH{\mathcal{H}}
\newcommand\GreenD{\mathcal{D}}
\newcommand\GreenJ{\mathcal{J}}
\newcommand\OrdL{\mathrel{\leq_\mathcal{L}}}
\newcommand\OrdR{\mathrel{\leq_\mathcal{R}}}
\newcommand\OrdJ{\mathrel{\leq_\mathcal{J}}}
\begin{document}
\title{NP-Completeness in the Gossip Monoid}

\maketitle

\begin{center}
PETER FENNER\footnote{Email \texttt{Peter.Fenner@postgrad.manchester.ac.uk}. Peter Fenner's research is supported by an EPSRC Doctoral Training Award.}, MARIANNE JOHNSON\footnote{Email \texttt{Marianne.Johnson@maths.manchester.ac.uk}.}  and MARK KAMBITES\footnote{Email \texttt{Mark.Kambites@manchester.ac.uk}.}

    \medskip

    School of Mathematics, \ University of Manchester, \\
    Manchester M13 9PL, \ England.

 \date{\today}

\thanks{}
\end{center}

\begin{abstract}
Gossip monoids form an algebraic model of networks with exclusive,
transient connections in which nodes, when they form a connection,
exchange all
known information. They also arise naturally in pure mathematics, as
the monoids generated by the set of all equivalence relations on a
given finite set under relational composition. We prove that a
number of important decision problems for these monoids (including
the membership problem, and hence the problem of deciding whether
a given state of knowledge can arise in a network of the kind under
consideration) are NP-complete. As well as being of interest in their
own right, these results shed light on the apparent difficulty of
establishing the cardinalities of the gossip monoids: a problem which
has attracted some attention in the last few years.
\end{abstract}

\section{Introduction}

\textit{Gossip problems} are concerned with the flow of information
through networks with exclusive, transient connections in which nodes,
when connected, exchange all known information.
The \textit{gossip monoid} (of rank $n$) forms an algebraic model of such
a network (with $n$ nodes). It is a semigroup whose elements correspond
to possible states of knowledge across the network, and in which the
natural multiplication action of a generating set simulates information
flow through the establishment of connections.

Gossip monoids are also of considerable interest for purely mathematical
reasons: they
arise naturally in semigroup theory as the monoids generated by the
set of equivalence relations on a finite set, under the operation of
relational composition. Since equivalence
relations are idempotent, this means gossip monoids are a natural and
interesting family of \textit{idempotent-generated monoids}, the latter
being an area of great current interest in semigroup theory
(see for example \cite{BriMarMea, GraRus1, GraRus2, DolGra, YanDolGou}.)

Despite the obvious importance of gossip monoids, both within semigroup
theory and for applications, remarkably little is known about them. In
the 1970s, a number of authors, including Tijdeman \cite{Tijdeman}, Baker and Shostak \cite{BakSho}, Hajnal, Milner and Szemeredi \cite{HajMilSze}, independently computed the minimum
number of two-way connections required to ensure permeation of all information
throughout the network; in semigroup-theoretic terms this is the length
of the zero element as a word with respect to a particularly natural set of generators. More recently, Brouwer,
Draisma and Frenk \cite{BroDraFre,Frenk} have introduced a continuous (tropical) analogue, termed the
\textit{lossy gossip monoid}, which has some intriguing connections to
composition of metrics and the tropical geometry of the orthogonal group.

Perhaps the most obvious question in this area is: how many elements are
there in the gossip monoid of rank $n$? In other words, how many distinct
knowledge configurations can arise in an $n$-node network of the type
discussed above? In \cite{BroDraFre}
this number was calculated computationally for $n \leq 9$, but no obvious
pattern emerges from their results, and the hope of extending the sequence
much further by similar brute-force
computations seems remote. It remains open whether there is an explicit formula for, or
even a markedly faster way of counting, the cardinality of the gossip monoid of
rank $n$.

In this paper, we consider the complexity of decision problems concerning
gossip monoids. Elements of gossip monoids are naturally represented
by boolean matrices, and an important problem is to decide, given a boolean
matrix, whether it represents a gossip element. One of our main results
is that this problem is NP-complete. As well as being of interest in its
own right, this sheds some light on the apparent difficulty of determining
the cardinality of the gossip monoid, by suggesting that there is no simple
combinatorial characterisation of elements, of the kind which might be
used to count them. We also establish NP-completeness for a number of
other important problems, including Green's $\mathcal{J}$-order.

Besides this introduction, the paper is organised into five sections.
Section~\ref{sec_gossip} recalls the definition and basic properties
of gossip monoids. Section~\ref{sec_problems} introduces the algorithmic
problems we study, and describes in outline our strategy for
establishing their complexity. Sections 4-6 are concerned with the proofs
of NP-hardness for these problems.

\section{The Gossip Monoid}\label{sec_gossip}
In this section we briefly introduce the gossip monoids and their
relationship to knowledge distribution in networks, and discuss some semigroup-theoretic properties of these monoids.

Consider $n$ people (``gossips'') each of whom knows a unique piece of information (a ``scandal'') initially unknown to the others. The people communicate by telephone and in each call the two participants tell each other every scandal they know. The `gossip problem' (what is the minimum number of calls required before every person knows every scandal?) attracted the attention of a number of researchers in the 1970s, including Tijdeman \cite{Tijdeman}, Baker and Shostak \cite{BakSho}, Hajnal, Milner and Szemeredi \cite{HajMilSze}, who proved in a variety of different ways that the minimal number of calls required is:
\begin{align*}
0 \;\text{ if }\; n = 1, \qquad
1 \;\text{ if }\; n = 2, \qquad
3 \;\text{ if }\; n = 3, \qquad
\text{ and } 2n-4 \;\text{ if }\; n \geq 4.
\end{align*}
Whilst the gossip problem concerns only the most efficient means to create one particular state of knowledge between the participants, it is clear that understanding the `gossip state-space' (in other words all possible states of knowledge that can occur in this system) is both a more complicated and more important problem; it is equivalent to understanding the distribution of knowledge within an $n$-node communication network with transient connections, assuming that initially each node possesses a unique piece of information, and that whenever a connection is made nodes exchange all known information.

We begin by showing that this system can be modelled by considering the right action of a particular monoid (the gossip monoid) on the set of all $n \times n$ boolean matrices, which we use to record these states of knowledge. We write $\mathbb{B}$ for the \textit{boolean semiring}, that is, the algebraic
structure comprising the set $\lbrace 0, 1 \rbrace$ equipped with the
operations of maximum (logical ``or'') and multiplication (logical ``and''.)
Algebraically, this structure is somewhat like a commutative ring, with maximum playing the role of addition, the fundamental difference being that the addition operation is non-invertible. The set $\mathbb{B}_n$ consisting of all $n \times n$ boolean matrices (that
is, matrices with entries from $\mathbb{B}$) forms a monoid under the matrix
multiplication induced from the operations in $\mathbb{B}$. (In other words, for $A, B \in \mathbb{B}_n$ the $(i,j)$th entry of the product $AB$ is obtained by taking the \emph{maximum} over $k$ of the products $a_{i,k}b_{k,j}$. It is easy to see that the `usual' identity matrix also behaves as an identity with respect to this new multiplication.) Moreover, there is a natural partial order on $\mathbb{B}_n$ induced from the natural order on $\mathbb{B}$; for $A, B \in \mathbb{B}_n$ we say that $A \preceq B$ if and only if $a_{i,j} \leq b_{i,j}$ for all indices $i,j$. We will then write $A \prec B$ to mean $A \preceq B$ but $A \neq B$.

Each matrix in $\mathbb{B}_n$ can be thought of as representing a state of knowledge distribution within
an $n$-node network: the entry in row $i$ and column $j$ being $1$ exactly if node $j$
(or ``person $j$'', in gossip problem terminology) has learnt the knowledge initially possessed by node $i$ (``scandal $i$''.) In particular, the identity matrix $I_n$ corresponds to the initial state of knowledge (each node knowing only what it knows initially.)

For $i, j \in \lbrace 1, \dots, n \rbrace$ the \textit{call matrix} $C[i,j]$ is
the matrix with $1$s on the main diagonal, and also in the $(i,j)$ and $(j,i)$
positions, and $0$s elsewhere. It is readily verified that if $K \in \mathbb{B}_n$ represents a state of knowledge,
then the product $K C[i,j]$ represents the state of knowledge resulting by
starting in the state represented by $K$ and exchanging all information (a ``call'') between node
$i$ and node $j$. The \textit{gossip monoid} $G_n$ of rank $n$ is the submonoid of $\mathbb{B}_n$ generated
by the set of all call matrices,
$$\lbrace C[i,j] \mid i, j \in \lbrace 1, \dots, n \rbrace \rbrace.$$
Its elements are exactly those matrices modelling states of knowledge which can
actually arise in a network of the kind under consideration, starting from
the initial state of knowledge and proceeding through calls in which two
nodes exchange all information available to them.

In matrix terms, the effect of right multiplication by $C[i,j]$
is to replace columns $i$ and $j$ with their element-wise maximum. Notice
in particular that this action is \textit{monotonic}, in the sense that $K \preceq KC[i,j]$; this corresponds to the fact that nodes
do not forget things once learnt. (Dually, one can see that the effect of \emph{left} multiplication by $C[i,j]$ is to replace \emph{rows} $i$ and $j$ with their element-wise maximum, and so it is clear that this left action is also monotonic, that is $K \preceq C[i,j]K$.) Notice also that this multiplication respects the partial order, in that $K \preceq L$ implies $KC[i,j] \preceq LC[i,j]$ and $C[i,j]K \preceq C[i,j]L$. Repeated use of these two properties gives us $AC \preceq ABC$ for all $A,B,C \in G_n$.

Recall that \emph{Green's relations} \cite{Green,CliPre} are five equivalence relations ($\GreenR$, $\GreenL$, $\GreenH$, $\GreenD$ and $\GreenJ$) and three pre-orders ($\OrdR$, $\OrdL$ and $\OrdJ$) which can be defined on any monoid (or semigroup) and play a fundamental role in understanding its structure. In the case of $G_n$, the five equivalence relations are
easily seen to be trivial --- we say that $G_n$ is a \textit{$\GreenJ$-trivial} monoid --- but the pre-orders remain important and we recall their definitions. We define $A \OrdR B$ [respectively,
$A \OrdL B$] if there exists a $C \in G_n$ with $A = BC$ [respectively, $A = CB$], and $A \OrdJ B$ if there exists $D, E \in G_n$ with $A = DBE$.
Notice that $A \OrdR B$ if and only if the state of knowledge represented by matrix $A$ can be obtained by starting with the state of knowledge represented by matrix $B$ and applying a sequence of calls.

It is straightforward to verify (as we shall explain below) that the call matrices satisfy the
following relations for all distinct values $i,j,k,l$:
\begin{eqnarray}
\label{idmpt}C[i,j] \ C[i,j] &\ = \ & C[i,j]\\
\label{commute}C[i,j] \ C[k,l]&\ = \ & C[k,l]C[i,j]\\
\label{transfer}C[i,j] \ C[j,k] \ C[i,j]& \ = \ & C[j,k] \ C[i,j] \ C[j,k].
\end{eqnarray}
In terms of gossip, the first set of relations simply record the fact if two people make a repeat
call when nothing has occurred in the interim then no information is exchanged in the second call. The second set of relations corresponds to the fact that if the two pairs of callers $\{i,j\}$ and $\{k,l\}$ are disjoint, then it doesn't matter in what order the calls between these pairs take place. (In the situation we are modelling, these communications could in fact occur concurrently.) The third set of relations, which are a kind of \textit{braid relation}, can easily be verified by matrix multiplication.

It follows immediately from the relations above that $G_n$ is the homomorphic image of an infinite 0-Hecke monoid (see \cite{MazSte}) corresponding to the Coxeter presentation with $n\choose 2$ involutions $c_{i,j}$ satisfying relations analogous to \eqref{commute} and \eqref{transfer} above. We also note that the \emph{double Catalan monoid} $DC_n$ studied by Mazorchuk and Steinberg \cite{MazSte} is the submonoid of $G_n$ generated by the call matrices $C[i,i+1]$ for $i=1, \ldots, n-1$; in fact the latter can be thought of as an algebraic model of \emph{linear} networks with exclusive, transient connections in which nodes, when they form a connection, exchange all known information.

The set of all call matrices is a particularly natural set of idempotent generators, which  may be generalised as follows. If $S \subseteq \lbrace 1, \dots, n \rbrace$ then the \textit{conference call matrix} $C[S]$ is the matrix with $1$s on the main diagonal and in the $(i,j)$ position for all $i, j \in S$, and $0$s elsewhere. In terms of knowledge distribution, right multiplication by $C[S]$ models a complete exchange of knowledge between the nodes in $S$; in matrix terms it replaces each column whose index is in $S$ with the element-wise maximum of all columns whose indices are in $S$. (The action of \emph{left} multiplication by $C[S]$, corresponding to replacing each \emph{row} whose index is in $S$ with the maximum of all
\emph{rows} whose indices are in $S$, is algebraically dual but less easy to visualise in terms of networks.) It is easy to see that every conference
call matrix is an idempotent element of $G_n$; indeed it follows from the solution to the original gossip problem
\cite{BakSho,HajMilSze,Tijdeman} that $C[S]$ is a product of $3$ call matrices if $|S| = 3$ and $2 |S|-4$ call
matrices if $|S| \geq 4$, these numbers of call matrices being the minimum possible.

We may associate to each element of $\mathbb{B}_n$ the binary relation on the set $\lbrace 1, \dots, n \rbrace$ of which it is the adjacency matrix, that is, the relation where $i$ is related to $j$ if and only if the $(i,j)$ entry of the matrix is $1$; matrix multiplication in $\mathbb{B}_n$ then corresponds to \textit{relational composition} of binary relations. It is easy to see that every call matrix (and indeed every conference call matrix)  is the adjacency matrix of an equivalence relation. In fact, equivalence relations correspond exactly to \textit{idempotent} elements of $G_n$:

\begin{proposition}
\label{Prop}
The idempotents in $G_n$ are precisely the matrices in $\mathbb{B}_n$ which correspond to equivalence relations on $\{1,\ldots, n\}$. (In other words, $A \in G_n$ is an idempotent if and only if there exists an equivalence relation $\sim$ on $\{1, \ldots, n\}$ such that $a_{i,j} = 1 \Leftrightarrow i \sim j$.)
\end{proposition}

\begin{proof}
First let $A \in G_n$ be an idempotent and let $\sim$ denote the binary relation on $\{1, \ldots, n\}$ given by $i \sim j$ if and only if $a_{i,j} = 1$. Since $A$ is a product of call matrices, it is clear that $a_{i,i}=1$ and hence $i \sim i$ for all $i$. Since $A$ is idempotent we note that $a_{i,j} = \max_k\{a_{i,k}a_{k,j}\}$, and hence whenever $i\sim k$ and $k \sim j$, we must also have $i\sim j$. To see that $\sim$ is symmetric we note that $A$ can be written as a product of call matrices, $A = A_1 \cdots A_m$ for some $m$, and since $A$ is idempotent we have $A^m = A$. We can therefore write $A$ as a product which contains $A_m, \ldots , A_1$ as a scattered subsequence. Since right multiplication by call matrix $C[i,j]$ replaces columns $i$ and $j$ with their maximum and left multiplication by the same matrix replaces rows $i$ and $j$ with their maximum, we see that $A_m \cdots A_1 = A^T$. We therefore have $A^T \preceq A$, and so if $i \sim j$ then $j \sim i$.

For the converse, let $\sim$ be an equivalence relation on $\{1, \ldots, n\}$ and let $A$ be the corresponding adjacency matrix. It is easy to see that $A$ is the product (in any order) of the conference call matrices corresponding to the equivalence classes of $\sim$, so that  $A \in G_n$. It remains to show that $A$ is an idempotent. It follows from reflexivity of $\sim$ that $B := A^2 \succeq A$, since $b_{i,j} \geq a_{i,j}a_{j,j}$. Finally we note that if $b_{i,j}=1$, then there exists $k$ such that $a_{i,k}a_{k,j}=1$. By transitivity of $\sim$ we conclude that $a_{i,j}=1$, hence showing that $A^2=A$.\end{proof}

\begin{corollary}
$G_n$ is the submonoid of $\mathbb{B}_n$ generated by the adjacency matrices of equivalence relations.
\end{corollary}
\begin{proof}
One containment follows from the fact that call matrices (which by definition generate $G_n$) are adjacency matrices of equivalence relations, and the other from
the fact that, by Proposition~\ref{Prop}, $G_n$ contains all adjacency matrices of equivalence relations.
\end{proof}

Throughout this paper, when the sizes are understood we write $1$ for the matrix consisting entirely of ones, $0$ for the zero matrix and $I_n$ for the identity matrix. Similarly, we shall write $\underline{0}$ to denote the vector of zeros and $\underline{1}$ to denote the vector of all ones.

\section{Algorithmic Problems in the Gossip Monoid}\label{sec_problems}
We shall consider the following decision problems concerning the
gossip monoids:
\begin{itemize}
\item the \emph{Gossip Membership Problem} (GMP) is to determine whether a given
$A \in \mathbb{B}_n$ is a member of $G_n$;
\item the \emph{Gossip $\mathcal{J}$-Order Problem} (GJP) is to determine,
given matrices $X, Y \in G_n$, whether $X \leq_\mathcal{J} Y$, that is, whether there exist matrices $U,V \in G_n$ with $UYV = X$;
\item the \emph{Gossip Transformation Problem} (GTP) is to determine,
given matrices $X, Y \in \mathbb{B}_n$, whether there exists a matrix $G \in G_n$
such that $X G = Y$;
\item the \emph{Maximal Gossip Transformation Problem} (MGTP) is the
restriction of GTP to pairs of matrices $X$, $Y$ where the matrix $X$
satisfies the \emph{maximal column condition}: $X$ is non-zero and the set of distinct columns of $X$ form an anti-chain. (In other words, every column is non-zero and maximal among the set of columns.)

\end{itemize}

Of these problems, the first three are of clear importance in their
own right.
GMP is critical for understanding the gossip monoid, and also has an
obvious application to deciding whether a given knowledge configuration
is (absolutely) reachable in a network. GTP corresponds to understanding orbits under
the natural action of $G_n$ on $\mathbb{B}_n$, and hence whether
one given knowledge configuration is reachable from another. The
$\mathcal{J}$-order is the key to understanding the ideal structure of a
$\mathcal{J}$-trivial monoid, and so GJP is essential for understanding
the semigroup-theoretic structure of $G_n$.
The final problem, MGTP, is less natural but is included because it functions
as a stepping stone in the proof of NP-hardness for GMP.

By monotonicity, and using the fact that there are only $n(n-1)$ zeros in the identity matrix, we see that there is a polynomial bound on the length of elements of $G_n$ as a product of the (quadratically many)
call generators. (In
\cite{BroDraFre} it is shown that this polynomial bound can be lowered to $n(n-1)/2$.) It is therefore possible for a non-deterministic
polynomial time computation to guess an element of $G_n$. This
clearly suffices to show that all of the above problems are in NP.

The remainder of the article is therefore concerned with establishing NP-hardness.
We first show that MGTP (and hence also GTP) is NP-hard, by a relatively
straightforward reduction from the \textit{Dominating Set Problem}. Then we show that GJP is NP-hard by a reduction from GTP. Finally, GMP is shown to be NP-hard by a (rather more complex) reduction from MGTP.

We briefly recall the definition of the Dominating Set Problem. Given an undirected graph $H$ with vertex set $V$ and edge set
$E \subseteq V \times V$, we say that $D \subseteq V$ is a
\emph{dominating set} for $H$ if every vertex is either in $D$
or adjacent to a vertex in $D$. Given a graph $H$ with vertex set $V$ and a positive integer
$k \leq |V|$, the Dominating Set Problem is to decide whether $H$
admits a dominating set of size at most $k$.
The problem is known to be NP-complete \cite[Problem GT2, Appendix A1]{GarJoh79} via a transformation from Vertex Cover. (See for example \cite[Theorem 2.8.6]{Jungnickel} for details.)

\section{NP-completeness of the Gossip Transformation Problems}

In this section we show that the Gossip Transformation Problem and the Maximal Gossip Transformation Problem are NP-complete. We show the NP-hardness of these problems with a polynomial time reduction from the Dominating Set Problem.

\begin{theorem}
\label{transthm}
GTP and MGTP are NP-complete.
\end{theorem}

\begin{proof}
We have already seen that both are in NP, and since MGTP is a restricted
version of GTP it suffices to show that MGTP is NP-hard. As discussed above,
we do this by reduction from the Dominating Set Problem.

Let $H$ be a graph, say with vertex set $V = \{1, \ldots, n\}$, and let
$1 \leq k \leq n$. We will construct matrices $A,B \in \mathbb{B}_{3n}$, with $A$ satisfying
the maximal column condition, such that $AG = B$ for some $G \in G_{3n}$
if and only if $H$ admits a dominating set of size at most $k$.

If $k = n$, then $V$ is a dominating set with size at most $k$, so we let $A = B = I_{2n}$. Otherwise, define a matrix
$M \in \mathbb{B}_n$ by
$$m_{i,j} = 1 \ \iff \ (i = j) \textrm{ or } (i \textrm{ and } j \textrm{ are adjacent in } H.)$$
It is easy to see that $D \subseteq \lbrace 1, \dots, n \rbrace$ is a
dominating set for $H$ if and only if
$\max\limits_{j \in D} (m_{i,j}) = 1$ for every $i$, so there exists
a dominating set of size at most $k$ if and only if there exists a set
of $k$ or fewer columns of $M$ whose maximum is $\underline{1}$.
Clearly, such a set of columns exists if and only if there is
a set of $k$ or fewer \emph{maximal} columns of $M$ whose maximum is
$\underline{1}$.

$M$ has at least one maximal column, $m$. Replace each non-maximal column of $M$ with $m$ and call the
resulting matrix $M'$. This ensures that $M'$ satisfies the maximal column condition but no columns are added which are not already columns of $M$. Note that finding the non-maximal column
vectors only requires $n(n-1)/2$ vector comparisons, so $M'$ can be
computed in polynomial time. Let
\begin{center}
\begin{tikzpicture}[
style1/.style={
  matrix of math nodes,
  every node/.append style={text width=#1,align=center,minimum height=5ex},
  nodes in empty cells,
  left delimiter=[,
  right delimiter=],
  },
style2/.style={
  matrix of math nodes,
  every node/.append style={text width=#1,align=center,minimum height=5ex},
  nodes in empty cells,
  left delimiter=\lbrace,
  right delimiter=\rbrace,
  }
]
\matrix[style1=0.30cm] (1mat)
{
  & & & & & \\
  & & & & & \\
  & & & & & \\
};
\draw[dashed]
  (1mat-1-1.south west) -- (1mat-1-6.south east);
\draw[dashed]
  (1mat-2-1.south west) -- (1mat-2-6.south east);
\draw[dashed]
  (1mat-1-2.north east) -- (1mat-3-2.south east);
\draw[dashed]
  (1mat-1-4.north east) -- (1mat-3-4.south east);
\node[font=\Large]
  at (1mat-1-1.east) {$M'$};
\node[font=\Large]
  at (1mat-1-3.east) {$0$};
\node[font=\Large]
  at (1mat-1-5.east) {$0$};
\node[font=\Large]
  at (1mat-2-1.east) {$0$};
\node[font=\Large]
  at (1mat-2-3.east) {$1$};
\node[font=\Large]
  at (1mat-2-5.east) {$1$};
\node[font=\Large]
  at (1mat-3-1.east) {$0$};
\node[font=\Large]
  at (1mat-3-3.east) {$0$};
\node[font=\Large]
  at (1mat-3-5.east) {$0$};
\draw[decoration={brace,raise=7pt},decorate]
  (1mat-1-1.north west) --
  node[above=8pt] {$n$}
  (1mat-1-2.north east);
\draw[decoration={brace,raise=7pt},decorate]
  (1mat-1-3.north west) --
  node[above=8pt] {$n$}
  (1mat-1-4.north east);
\draw[decoration={brace,raise=7pt},decorate]
  (1mat-1-5.north west) --
  node[above=8pt] {$n$}
  (1mat-1-6.north east);
\draw[decoration={brace,raise=12pt},decorate]
  (1mat-1-6.north east) --
  node[right=15pt] {$n$}
  (1mat-1-6.south east);
\draw[decoration={brace,raise=12pt},decorate]
  (1mat-2-6.north east) --
  node[right=15pt] {$n$}
  (1mat-2-6.south east);
\draw[decoration={brace,raise=12pt},decorate]
  (1mat-3-6.north east) --
  node[right=15pt] {$n$}
  (1mat-3-6.south east);
\node[left=15pt] at (1mat-2-1.west) {$A = $};

\matrix[style1=0.40cm, right=5.25cm] (2mat)
{
  & & & & & \\
  & & & & & \\
  & & & & & \\
};
\draw[dashed]
  (2mat-1-1.south west) -- (2mat-1-6.south east);
\draw[dashed]
  (2mat-2-1.south west) -- (2mat-2-6.south east);
\draw[dashed]
  (2mat-1-2.north east) -- (2mat-3-2.south east);
\draw[dashed]
  (2mat-1-4.north east) -- (2mat-3-4.south east);
\draw[dashed]
  (2mat-1-5.north east) -- (2mat-1-5.south east);
\node[font=\Large]
  at (2mat-1-1.east) {$M'$};
\node[font=\Large]
  at (2mat-1-3.east) {$M'$};
\node[font=\Large]
  at (2mat-1-5) {$1$};
\node[font=\Large]
  at (2mat-1-6) {$0$};
\node[font=\Large]
  at (2mat-2-1.east) {$1$};
\node[font=\Large]
  at (2mat-2-3.east) {$1$};
\node[font=\Large]
  at (2mat-2-5.east) {$1$};
\node[font=\Large]
  at (2mat-3-1.east) {$0$};
\node[font=\Large]
  at (2mat-3-3.east) {$0$};
\node[font=\Large]
  at (2mat-3-5.east) {$0$};
\draw[decoration={brace,raise=7pt},decorate]
  (2mat-1-1.north west) --
  node[above=8pt] {$n$}
  (2mat-1-2.north east);
\draw[decoration={brace,raise=7pt},decorate]
  (2mat-1-3.north west) --
  node[above=8pt] {$n$}
  (2mat-1-4.north east);
\draw[decoration={brace,raise=7pt},decorate]
  (2mat-1-5.north west) --
  node[above=8pt] {$k\;$}
  (2mat-1-5.north east);
\draw[decoration={brace,raise=7pt},decorate]
  (2mat-1-6.north west) --
  node[above=8pt] {$\;\;\;n-k$}
  (2mat-1-6.north east);
\draw[decoration={brace,raise=12pt},decorate]
  (2mat-1-6.north east) --
  node[right=15pt] {$n$}
  (2mat-1-6.south east);
\draw[decoration={brace,raise=12pt},decorate]
  (2mat-2-6.north east) --
  node[right=15pt] {$n$}
  (2mat-2-6.south east);
\draw[decoration={brace,raise=12pt},decorate]
  (2mat-3-6.north east) --
  node[right=15pt] {$n$}
  (2mat-3-6.south east);
\node[left=15pt] at (2mat-2-1.west) {and \; $B = $};
\node[right=25pt] at (2mat-2-6.east) {.};
\end{tikzpicture}
\end{center}

Note that the matrix $A$ satisfies the maximal column condition since $M'$ does, and $A$ and $B$ give an instance of the problem $MGTP$.

If $H$ admits a dominating set of size $k$ or less then there is a set
of $k$ or fewer columns of $M'$ whose maximum is $\underline{1}$. We
can multiply matrix $A$ on the right
by call matrices which copy each of these columns into a unique
column between $2n+1$ and $2n+k$. These call matrices will not affect the first $n$ rows of columns $1, \dots, n$ but will put a 1 in each of the first $n$ rows in columns $2n+1, \dots, 2n+k$. Multiplying this product on the right by the conference
call matrix $C[\lbrace 2n+1, \dots, 2n+k\rbrace]$ and by each of the call matrices $C[i, n+i]$ for $1 \leq i \leq n$, in any order, will therefore result in matrix $B$.
The product of these call matrices is thus an element $G \in G_{3n}$ such that
$AG = B$.

Conversely, if there is a $G \in G_{3n}$ such that $AG = B$, then there
is a sequence of call matrices whose product is $G$, say $G = C_1 C_2 \cdots C_q$.
Without loss of generality we may clearly assume that $A C_1 \cdots C_p \prec A C_1 \cdots C_p C_{p+1}$ for all $p$.

Since right multiplication by a call matrix has the effect of replacing two columns with their maximum, it is clear that each column of each product $A C_1 \cdots C_p$ must be equal to the maximum of some collection of columns of $A$. In particular, the first $n$ rows of each column must equal the maximum of some collection of columns of $M'$. In the product $B = A C_1 \cdots C_q$ the first $n$ rows of column $n+i$, for $1 \leq i \leq n$, equal the $i$th column of $M'$. By monotonicity and the fact that $M'$ satisfies the maximal column condition, the first $n$ rows of column $n+i$ must equal either $0$ or the $i$th column of $M'$ for every product $A C_1 \cdots C_p$.

\begin{claimnonum}
For each $r \in \{2n + 1, \ldots, 2n + k\}$, there is at most one matrix $C_p$ of the form $C[i,r]$ such that $i \leq 2n$ and the two products $A C_1 \cdots C_{p-1}$ and $A C_1 \cdots C_p$ differ on column $r$.
\end{claimnonum}

\begin{claimproof}
Assume for contradiction that the sequence $C_1,\dots,C_q$ contains
call matrices $C_s = C[i,r]$ and $C_t = C[j,r]$, in that order, with $i,j \leq 2n$ and $2n+1 \leq r \leq 2n+k$, such that the products $AC_1\cdots C_{s-1}$ and $AC_1\cdots C_s$ differ on column $r$ and the products $AC_1\cdots C_{t-1}$ and $AC_1\cdots C_t$ also differ on column $r$.

We observe that the last $2n$ rows of column $r$ are identical in $A$ and $B$, so by monotonicity the differences in column $r$ occur in the first $n$ rows. Let $v_s, v_t \in \mathbb{B}^n$ be the first $n$ rows of column $r$ in the matrices $AC_1 \cdots C_s$ and $AC_1 \cdots C_t$, respectively. Since the differences occur in the first $n$ rows, and by monotonicity, we have $0 \prec v_s \prec v_t$. As the last call matrix in the product $AC_1 \cdots C_s$ is $C_s = C[i,r]$, columns $i$ and $r$ must be identical in this matrix, and so $v_s$ is also equal to the first $n$ rows of column $i$. If $i \leq n$ then $v_s$ is equal to column $i$ of $M'$. Otherwise $n < i \leq 2n$ and since $v_s$ is non-zero it is equal to the $(i-n)$th column of $M'$. Either way, $v_s$ is equal to a column of $M'$ and similar observations show that $v_t$ is also equal to a column of $M'$. Now $v_s \prec v_t$ contradicts the maximal column condition on $M'$, and this proves the claim.
\end{claimproof}
\medskip

It follows from the claim that right-multiplying $A$ by the product
$C_1 \cdots C_q$ copies at most $k$ of the columns of $M'$
into the columns $2n+1, \dots, 2n+k$. But since the result of this right multiplication
is to place $1$s in the first $n$ rows of these columns, it must be that
$\underline{1}$ is a maximum of at most $k$ columns of $M'$. It follows
from our observations that $H$ admits a dominating set of size at most $k$.\end{proof}

\section{NP-Completeness of the Gossip $\mathcal{J}$-order Problem}

To show GJP is NP-hard we will give a polynomial time reduction from
GTP. This reduction uses the following lemma, which allows us to ``nest'' an arbitrary
boolean matrix inside a gossip matrix with polynomially larger size:

\begin{lemma}\label{boolingoss}
Let $n\geq 2$. For any $A \in \mathbb{B}_n$, the following matrix lies in $G_{n(n+1)}$:
\begin{center}
\begin{tikzpicture}[
style1/.style={
  matrix of math nodes,
  every node/.append style={text width=#1,align=center,minimum height=5ex},
  nodes in empty cells,
  left delimiter=[,
  right delimiter=],
  },
style2/.style={
  matrix of math nodes,
  every node/.append style={text width=#1,align=center,minimum height=5ex},
  nodes in empty cells,
  left delimiter=\lbrace,
  right delimiter=\rbrace,
  }
]
\matrix[style1=0.65cm] (1mat)
{
  & & \\
  & & \\
  & & \\
};
\draw[dashed]
  (1mat-1-1.south west) -- (1mat-1-3.south east);
\draw[dashed]
  (1mat-1-2.north east) -- (1mat-3-2.south east);
\node[font=\Large]
  at (1mat-1-2.west) {$1$};
\node[font=\Large]
  at (1mat-1-3) {$A$};
\node[font=\Large]
  at (1mat-3-2.north west) {$1$};
\node[font=\Large]
  at (1mat-3-3.north) {$1$};
\draw[decoration={brace,raise=7pt},decorate]
  (1mat-1-1.north west) --
  node[above=8pt] {$n^{2}$}
  (1mat-1-2.north east);
\draw[decoration={brace,raise=7pt},decorate]
  (1mat-1-3.north west) --
  node[above=8pt] {$n$}
  (1mat-1-3.north east);
\draw[decoration={brace,raise=12pt},decorate]
  (1mat-1-3.north east) --
  node[right=15pt] {$n$}
  (1mat-1-3.south east);
\draw[decoration={brace,raise=12pt},decorate]
  (1mat-2-3.north east) --
  node[right=15pt] {$n^{2}$}
  (1mat-3-3.south east);
\node[left=15pt] at (1mat-2-1.west) {$X = $};

\end{tikzpicture}
\end{center}
\end{lemma}

\begin{proof}
We shall show that $X$ can be written as $X = X_1 X_2 X_3 X_4$, where the matrices $X_i$ are defined in terms of conference call matrices as follows
\begin{eqnarray*}
X_1 = C[\{n+1, \ldots, n(n+1)\}],\qquad\;\;\;\;\;
&& X_4 = C[\{1, \ldots, n^2\}],\\
X_2 = \prod_{i = 1}^n C[\{i + n(j-1) : 1 \leq j \leq n\}],
&& X_3 = \prod_{j = 1}^n C[\{i + n(j-1) : a_{i,j} = 1\} \cup \{n^2 + j\}], \qquad
\end{eqnarray*}
and where the products taken in $X_2$ and $X_3$ can be in any order.

Each matrix in the product $X_2$ is of the form $C[\{i + n(j-1) : 1 \leq j \leq n\}]$ for some $i$. This matrix has a $1$ in the $a,b$ position if and only if either $a=b$, or $a,b \leq n^2$ and $a$ and $b$ are both congruent to $i$ modulo $n$. The product $X_2$ therefore has a $1$ in the $a,b$ position if and only if either $a=b$, or $a,b \leq n^2$ and $a$ and $b$ are congruent modulo $n$. The result is that the top-left $n^2 \times n^2$ block of $X_2$ consists of an $n \times n$ array of copies of $I_n$:
\begin{center}
\begin{tikzpicture}[
style1/.style={
  matrix of math nodes,
  every node/.append style={text width=#1,align=center,minimum height=5ex},
  nodes in empty cells,
  left delimiter=[,
  right delimiter=],
  },
style2/.style={
  matrix of math nodes,
  every node/.append style={text width=#1,align=center,minimum height=5ex},
  nodes in empty cells,
  left delimiter=\lbrace,
  right delimiter=\rbrace,
  }
]
\matrix[style1=0.65cm] (1mat)
{
  & & & \\
  & & & \\
  & & & \\
  & & & \\
};
\draw[dashed]
  (1mat-1-1.south west) -- (1mat-1-3.south east);
\draw[dashed]
  (1mat-3-1.south west) -- (1mat-3-4.south east);
\draw[dashed]
  (1mat-1-1.north east) -- (1mat-3-1.south east);
\draw[dashed]
  (1mat-1-3.north east) -- (1mat-4-3.south east);
\draw[dashed]
  (1mat-1-2.north east) -- (1mat-3-2.south east);
\draw[dashed]
  (1mat-2-1.south west) -- (1mat-2-3.south east);
\node[font=\Large]
  at (1mat-1-1) {$I_n$};
\node[font=\Large]
  at (1mat-1-2) {\ldots};
\node[font=\Large]
  at (1mat-1-3) {$I_n$};
\node[font=\Large]
  at (1mat-2-1) {$\vdots$};
\node[font=\Large]
  at (1mat-2-2) {$\ddots$};
\node[font=\Large]
  at (1mat-2-3) {$\vdots$};
\node[font=\Large]
  at (1mat-2-4) {$0$};
\node[font=\Large]
  at (1mat-3-1) {$I_n$};
\node[font=\Large]
  at (1mat-3-2) {$\ldots$};
\node[font=\Large]
  at (1mat-3-3) {$I_n$};
\node[font=\Large]
  at (1mat-4-2) {$0$};
\node[font=\Large]
  at (1mat-4-4) {$I_n$};
\draw[decoration={brace,raise=7pt},decorate]
  (1mat-1-1.north west) --
  node[above=8pt] {$n^2$}
  (1mat-1-3.north east);
\draw[decoration={brace,raise=7pt},decorate]
  (1mat-1-4.north west) --
  node[above=8pt] {$n$}
  (1mat-1-4.north east);
\draw[decoration={brace,raise=12pt},decorate]
  (1mat-1-4.north east) --
  node[right=15pt] {$n^2$}
  (1mat-3-4.south east);
\draw[decoration={brace,raise=12pt},decorate]
  (1mat-4-4.north east) --
  node[right=15pt] {$n$}
  (1mat-4-4.south east);
\node[left=15pt] at (1mat-3-1.north west) {$X_2 = $};
\node[right=30pt] at (1mat-3-4.north east) {$.$};

\end{tikzpicture}
\end{center}

For $1 \leq j \leq n$ let $S_j = \{i : a_{i,j} = 1\} \subseteq \{1, \ldots, n\}$ and let $A_{[j]}$ be the matrix which is the same as $A$ on column $j$ but zero elsewhere. For each $j$ we have
\begin{center}
\begin{tikzpicture}[
style1/.style={
  matrix of math nodes,
  every node/.append style={text width=#1,align=center,minimum height=6ex},
  nodes in empty cells,
  left delimiter=[,
  right delimiter=],
  },
style2/.style={
  matrix of math nodes,
  every node/.append style={text width=#1,align=center,minimum height=5ex},
  nodes in empty cells,
  left delimiter=\lbrace,
  right delimiter=\rbrace,
  }
]
\matrix[style1=1cm] (1mat)
{
  & & & \\
  & & & \\
  & & & \\
  & & & \\
};
\draw[dashed]
  (1mat-1-1.south west) -- (1mat-1-2.south east);
\draw[dashed]
  (1mat-1-4.south west) -- (1mat-1-4.south east);
\draw[dashed]
  (1mat-2-2.south west) -- (1mat-2-4.south east);
\draw[dashed]
  (1mat-3-1.south west) -- (1mat-3-4.south east);
\draw[dashed]
  (1mat-1-1.north east) -- (1mat-2-1.south east);
\draw[dashed]
  (1mat-4-1.north east) -- (1mat-4-1.south east);
\draw[dashed]
  (1mat-2-2.north east) -- (1mat-4-2.south east);
\draw[dashed]
  (1mat-1-3.north east) -- (1mat-4-3.south east);

\node
  at (1mat-1-1) {$I_{n(j-1)}$};
\node[font=\Large]
  at (1mat-1-3) {$0$};
\node[font=\large]
  at (1mat-1-4) {$0$};
\node
  at (1mat-2-2) {$C[S_j]$};
\node
  at (1mat-2-4) {$A_{[j]}$};
\node[font=\Large]
  at (1mat-3-1) {$0$};
\node
  at (1mat-3-3) {$\;I_{n(n-j)}$};
\node[font=\large]
  at (1mat-3-4) {$0$};
\node[font=\large]
  at (1mat-4-1) {$0$};
\node
  at (1mat-4-2) {$(A_{[j]})^T$};
\node[font=\large]
  at (1mat-4-3) {$0$};
\node
  at (1mat-4-4) {$I_n$};
\draw[decoration={brace,raise=7pt},decorate]
  (1mat-1-1.north west) --
  node[above=8pt] {$n(j-1)$}
  (1mat-1-1.north east);
\draw[decoration={brace,raise=7pt},decorate]
  (1mat-1-2.north west) --
  node[above=8pt] {$n$}
  (1mat-1-2.north east);
\draw[decoration={brace,raise=7pt},decorate]
  (1mat-1-3.north west) --
  node[above=8pt] {$n(n-j)$}
  (1mat-1-3.north east);
\draw[decoration={brace,raise=7pt},decorate]
  (1mat-1-4.north west) --
  node[above=8pt] {$n$}
  (1mat-1-4.north east);
\draw[decoration={brace,raise=12pt},decorate]
  (1mat-1-4.north east) --
  node[right=15pt] {$n(j-1)$}
  (1mat-1-4.south east);
\draw[decoration={brace,raise=12pt},decorate]
  (1mat-2-4.north east) --
  node[right=15pt] {$n$}
  (1mat-2-4.south east);
\draw[decoration={brace,raise=12pt},decorate]
  (1mat-3-4.north east) --
  node[right=15pt] {$n(n-j)$}
  (1mat-3-4.south east);
\draw[decoration={brace,raise=12pt},decorate]
  (1mat-4-4.north east) --
  node[right=15pt] {$n$}
  (1mat-4-4.south east);
\node[left=15pt] at (1mat-3-1.north west) {$C[\{n(j-1) + i : i \in S_j\} \cup \{n^2 + j\}] = $};
\node[right=55pt] at (1mat-3-4.north east) {$.$};

\end{tikzpicture}
\end{center}
Notice that, as $j$ varies, the corresponding conference call matrices are between disjoint sets of nodes. The product of all such matrices, as $j$ ranges between $1$ and $n$, is
therefore
\begin{center}
\begin{tikzpicture}[
style1/.style={
  matrix of math nodes,
  every node/.append style={text width=#1,align=center,minimum height=5.5ex},
  nodes in empty cells,
  left delimiter=[,
  right delimiter=],
  },
style2/.style={
  matrix of math nodes,
  every node/.append style={text width=#1,align=center,minimum height=5ex},
  nodes in empty cells,
  left delimiter=\lbrace,
  right delimiter=\rbrace,
  }
]
\matrix[style1=1cm] (1mat)
{
  & & & & \\
  & & & & \\
  & & & & \\
  & & & & \\
  & & & & \\
};
\draw[dashed]
  (1mat-1-1.south west) -- (1mat-1-2.south east);
\draw[dashed]
  (1mat-1-5.south west) -- (1mat-1-5.south east);
\draw[dashed]
  (1mat-2-2.south west) -- (1mat-2-3.south east);
\draw[dashed]
  (1mat-2-5.south west) -- (1mat-2-5.south east);
\draw[dashed]
  (1mat-3-3.south west) -- (1mat-3-5.south east);
\draw[dashed]
  (1mat-4-1.south west) -- (1mat-4-5.south east);
\draw[dashed]
  (1mat-1-1.north east) -- (1mat-2-1.south east);
\draw[dashed]
  (1mat-5-1.north east) -- (1mat-5-1.south east);
\draw[dashed]
  (1mat-2-2.north east) -- (1mat-3-2.south east);
\draw[dashed]
  (1mat-5-2.north east) -- (1mat-5-2.south east);
\draw[dashed]
  (1mat-3-3.north east) -- (1mat-5-3.south east);
\draw[dashed]
  (1mat-1-4.north east) -- (1mat-5-4.south east);

\node
  at (1mat-1-1) {$C[S_1]$};
\node[font=\Large]
  at (1mat-1-3.south east) {$0$};
\node
  at (1mat-1-5) {$A_{[1]}$};
\node
  at (1mat-2-2) {$C[S_2]$};
\node
  at (1mat-2-5) {$A_{[2]}$};
\node[font=\Large]
  at (1mat-3-1.south east) {$0$};
\node[font=\Large]
  at (1mat-3-3) {$\ddots$};
\node[font=\large]
  at (1mat-3-5) {$\vdots$};
\node
  at (1mat-4-4) {$C[S_n]$};
\node
  at (1mat-4-5) {$A_{[n]}$};
\node
  at (1mat-5-1) {$(A_{[1]})^T$};
\node
  at (1mat-5-2) {$(A_{[2]})^T$};
\node[font=\Large]
  at (1mat-5-3) {$\ldots$};
\node
  at (1mat-5-4) {$(A_{[n]})^T$};
\node[font=\Large]
  at (1mat-5-5) {$I_n$};
\draw[decoration={brace,raise=7pt},decorate]
  (1mat-1-1.north west) --
  node[above=8pt] {$n^2$}
  (1mat-1-4.north east);
\draw[decoration={brace,raise=7pt},decorate]
  (1mat-1-5.north west) --
  node[above=8pt] {$n$}
  (1mat-1-5.north east);
\draw[decoration={brace,raise=12pt},decorate]
  (1mat-1-5.north east) --
  node[right=15pt] {$n^2$}
  (1mat-4-5.south east);
\draw[decoration={brace,raise=12pt},decorate]
  (1mat-5-5.north east) --
  node[right=15pt] {$n$}
  (1mat-5-5.south east);
\node[left=15pt] at (1mat-3-1.west) {$X_3 = $};
\node[right=30pt] at (1mat-3-5.east) {.};

\end{tikzpicture}
\end{center}

Consider the effect of multiplying $X_3$ on the left by $X_2$. This leaves the last $n$ rows unchanged. The remaining $n^2$ rows can be split into $n$ blocks of $n$ rows each. The structure of the identity matrices in $X_2$ means that in $X_2 X_3$ these blocks are all identical to each other and equal to the element-wise maximum of the $n$ corresponding blocks in $X_3$. Since the element-wise maximum of $A_{[1]}$ to $A_{[n]}$ is $A$, we get

\begin{center}
\begin{tikzpicture}[
style1/.style={
  matrix of math nodes,
  every node/.append style={text width=#1,align=center,minimum height=5ex},
  nodes in empty cells,
  left delimiter=[,
  right delimiter=],
  },
style2/.style={
  matrix of math nodes,
  every node/.append style={text width=#1,align=center,minimum height=5ex},
  nodes in empty cells,
  left delimiter=\lbrace,
  right delimiter=\rbrace,
  }
]
\matrix[style1=1cm] (1mat)
{
  & & & & \\
  & & & & \\
  & & & & \\
  & & & & \\
  & & & & \\
};
\draw[dashed]
  (1mat-1-1.south west) -- (1mat-1-5.south east);
\draw[dashed]
  (1mat-2-1.south west) -- (1mat-2-5.south east);
\draw[dashed]
  (1mat-3-1.south west) -- (1mat-3-5.south east);
\draw[dashed]
  (1mat-4-1.south west) -- (1mat-4-5.south east);
\draw[dashed]
  (1mat-1-1.north east) -- (1mat-5-1.south east);
\draw[dashed]
  (1mat-1-2.north east) -- (1mat-5-2.south east);
\draw[dashed]
  (1mat-1-3.north east) -- (1mat-5-3.south east);
\draw[dashed]
  (1mat-1-4.north east) -- (1mat-5-4.south east);

\node
  at (1mat-1-1) {$C[S_1]$};
\node
  at (1mat-1-2) {$C[S_2]$};
\node[font=\Large]
  at (1mat-1-3) {$\ldots$};
\node
  at (1mat-1-4) {$C[S_n]$};
\node[font=\large]
  at (1mat-1-5) {$A$};
\node
  at (1mat-2-1) {$C[S_1]$};
\node
  at (1mat-2-2) {$C[S_2]$};
\node
  at (1mat-2-4) {$C[S_n]$};
\node[font=\large]
  at (1mat-2-5) {$A$};
\node[font=\Large]
  at (1mat-3-1) {$\vdots$};
\node[font=\Large]
  at (1mat-3-3) {$\ddots$};
\node[font=\large]
  at (1mat-3-5) {$\vdots$};
\node
  at (1mat-4-1) {$C[S_1]$};
\node
  at (1mat-4-2) {$C[S_2]$};
\node
  at (1mat-4-4) {$C[S_n]$};
\node[font=\large]
  at (1mat-4-5) {$A$};
\node
  at (1mat-5-1) {$(A_{[1]})^T$};
\node
  at (1mat-5-2) {$(A_{[2]})^T$};
\node[font=\Large]
  at (1mat-5-3) {$\ldots$};
\node
  at (1mat-5-4) {$(A_{[n]})^T$};
\node[font=\Large]
  at (1mat-5-5) {$I_n$};
\draw[decoration={brace,raise=7pt},decorate]
  (1mat-1-1.north west) --
  node[above=8pt] {$n^2$}
  (1mat-1-4.north east);
\draw[decoration={brace,raise=7pt},decorate]
  (1mat-1-5.north west) --
  node[above=8pt] {$n$}
  (1mat-1-5.north east);
\draw[decoration={brace,raise=12pt},decorate]
  (1mat-1-5.north east) --
  node[right=15pt] {$n^2$}
  (1mat-4-5.south east);
\draw[decoration={brace,raise=12pt},decorate]
  (1mat-5-5.north east) --
  node[right=15pt] {$n$}
  (1mat-5-5.south east);
\node[left=15pt] at (1mat-3-1.west) {$X_2 X_3 = $};
\node[right=30pt] at (1mat-3-5.east) {.};

\end{tikzpicture}
\end{center}

Between them, the last $n^2$ rows of $X_2 X_3$ contain a $1$ in each column (each conference call has only $1$s on its diagonal.) Therefore, multiplying on the left by $X_1 = C[\{n+1, \ldots, n(n+1)\}]$ fills each of these rows with $1$s. Now, between them, the first $n^2$ columns of $X_1 X_2 X_3$ contain a $1$ in each row, so multiplying on the right by $X_4 = C[\{1, \ldots, n^2\}]$ fills each of these columns with $1$s. We therefore obtain $X_1X_2X_3X_4=X$, as required.\end{proof}

\begin{theorem}
GJP is NP-complete.
\end{theorem}

\begin{proof}
We have already seen that the problem is in NP. We show the problem is NP-hard by reduction from MGTP. Given matrices $A, B \in \mathbb{B}_n$, let

\begin{center}
\begin{tikzpicture}[
style1/.style={
  matrix of math nodes,
  every node/.append style={text width=#1,align=center,minimum height=5ex},
  nodes in empty cells,
  left delimiter=[,
  right delimiter=],
  },
style2/.style={
  matrix of math nodes,
  every node/.append style={text width=#1,align=center,minimum height=5ex},
  nodes in empty cells,
  left delimiter=\lbrace,
  right delimiter=\rbrace,
  }
]
\matrix[style1=0.65cm] (1mat)
{
  & & & \\
  & & & \\
  & & & \\
  & & & \\
};
\draw[dashed]
  (1mat-2-1.south west) -- (1mat-2-4.south east);
\draw[dashed]
  (1mat-1-2.north east) -- (1mat-4-2.south east);
\draw[dashed]
  (1mat-1-3.south west) -- (1mat-1-4.south east);
\draw[dashed]
  (1mat-1-3.north east) -- (1mat-2-3.south east);
\node[font=\Large]
  at (1mat-1-2.south west) {$1$};
\node[font=\Large]
  at (1mat-1-3) {$A$};
\node[font=\Large]
  at (1mat-1-4) {$I_n$};
\node[font=\Large]
  at (1mat-2-3) {$0$};
\node[font=\Large]
  at (1mat-2-4) {$1$};
\node[font=\Large]
  at (1mat-3-2.south west) {$1$};
\node[font=\Large]
  at (1mat-3-3.south east) {$1$};
\draw[decoration={brace,raise=7pt},decorate]
  (1mat-1-1.north west) --
  node[above=8pt] {$(2n)^2$}
  (1mat-1-2.north east);
\draw[decoration={brace,raise=7pt},decorate]
  (1mat-1-3.north west) --
  node[above=8pt] {$n$}
  (1mat-1-3.north east);
\draw[decoration={brace,raise=7pt},decorate]
  (1mat-1-4.north west) --
  node[above=8pt] {$n$}
  (1mat-1-4.north east);
\draw[decoration={brace,raise=12pt},decorate]
  (1mat-1-4.north east) --
  node[right=15pt] {$n$}
  (1mat-1-4.south east);
\draw[decoration={brace,raise=12pt},decorate]
  (1mat-2-4.north east) --
  node[right=15pt] {$n$}
  (1mat-2-4.south east);
\draw[decoration={brace,raise=12pt},decorate]
  (1mat-3-4.north east) --
  node[right=15pt] {$(2n)^2$}
  (1mat-4-4.south east);
\node[left=15pt] at (1mat-3-1.north west) {$X = $};

\matrix[style1=0.65cm, right=5.5cm] (2mat)
{
  & & & \\
  & & & \\
  & & & \\
  & & & \\
};
\draw[dashed]
  (2mat-2-1.south west) -- (2mat-2-4.south east);
\draw[dashed]
  (2mat-1-2.north east) -- (2mat-4-2.south east);
\draw[dashed]
  (2mat-1-3.south west) -- (2mat-1-4.south east);
\draw[dashed]
  (2mat-1-3.north east) -- (2mat-2-3.south east);
\node[font=\Large]
  at (2mat-1-2.south west) {$1$};
\node[font=\Large]
  at (2mat-1-3) {$B$};
\node[font=\Large]
  at (2mat-1-4) {$I_n$};
\node[font=\Large]
  at (2mat-2-3) {$0$};
\node[font=\Large]
  at (2mat-2-4) {$1$};
\node[font=\Large]
  at (2mat-3-2.south west) {$1$};
\node[font=\Large]
  at (2mat-3-3.south east) {$1$};
\draw[decoration={brace,raise=7pt},decorate]
  (2mat-1-1.north west) --
  node[above=8pt] {$(2n)^2$}
  (2mat-1-2.north east);
\draw[decoration={brace,raise=7pt},decorate]
  (2mat-1-3.north west) --
  node[above=8pt] {$n$}
  (2mat-1-3.north east);
\draw[decoration={brace,raise=7pt},decorate]
  (2mat-1-4.north west) --
  node[above=8pt] {$n$}
  (2mat-1-4.north east);
\draw[decoration={brace,raise=12pt},decorate]
  (2mat-1-4.north east) --
  node[right=15pt] {$n$}
  (2mat-1-4.south east);
\draw[decoration={brace,raise=12pt},decorate]
  (2mat-2-4.north east) --
  node[right=15pt] {$n$}
  (2mat-2-4.south east);
\draw[decoration={brace,raise=12pt},decorate]
  (2mat-3-4.north east) --
  node[right=15pt] {$(2n)^2$}
  (2mat-4-4.south east);
\node[left=15pt] at (2mat-3-1.north west) {$\mathrm{and} \quad Y = $};
\node[right=30pt] at (2mat-3-4.north east) {.};
\end{tikzpicture}
\end{center}
These can clearly be constructed in polynomial time, and by Lemma \ref{boolingoss} we know that $X,Y \in G_{2n(2n+1)}$.

If there exists $G \in G_n$ such that $AG = B$ then a simple calculation shows that $UXV = Y$ (and hence $Y \leq_{\mathcal{J}} X$) where

\begin{center}
\begin{tikzpicture}[
style1/.style={
  matrix of math nodes,
  every node/.append style={text width=#1,align=center,minimum height=5ex},
  nodes in empty cells,
  left delimiter=[,
  right delimiter=],
  },
style2/.style={
  matrix of math nodes,
  every node/.append style={text width=#1,align=center,minimum height=5ex},
  nodes in empty cells,
  left delimiter=\lbrace,
  right delimiter=\rbrace,
  }
]
\matrix[style1=0.65cm] (1mat)
{
  & & & \\
  & & & \\
  & & & \\
  & & & \\
};
\draw[dashed]
  (1mat-2-1.south west) -- (1mat-2-4.south east);
\draw[dashed]
  (1mat-1-2.north east) -- (1mat-4-2.south east);
\draw[dashed]
  (1mat-3-3.south west) -- (1mat-3-4.south east);
\draw[dashed]
  (1mat-3-3.north east) -- (1mat-4-3.south east);
\node[font=\Large]
  at (1mat-1-1.south east) {$I_{(2n)^2}$};
\node[font=\Large]
  at (1mat-1-3.south east) {$0$};
\node[font=\Large]
  at (1mat-3-1.south east) {$0$};
\node[font=\Large]
  at (1mat-3-3) {$G$};
\node[font=\Large]
  at (1mat-3-4) {$0$};
\node[font=\Large]
  at (1mat-4-3) {$0$};
\node[font=\Large]
  at (1mat-4-4) {$I_n$};
\draw[decoration={brace,raise=7pt},decorate]
  (1mat-1-1.north west) --
  node[above=8pt] {$(2n)^2$}
  (1mat-1-2.north east);
\draw[decoration={brace,raise=7pt},decorate]
  (1mat-1-3.north west) --
  node[above=8pt] {$n$}
  (1mat-1-3.north east);
\draw[decoration={brace,raise=7pt},decorate]
  (1mat-1-4.north west) --
  node[above=8pt] {$n$}
  (1mat-1-4.north east);
\draw[decoration={brace,raise=12pt},decorate]
  (1mat-1-4.north east) --
  node[right=15pt] {$(2n)^2$}
  (1mat-2-4.south east);
\draw[decoration={brace,raise=12pt},decorate]
  (1mat-3-4.north east) --
  node[right=15pt] {$n$}
  (1mat-3-4.south east);
\draw[decoration={brace,raise=12pt},decorate]
  (1mat-4-4.north east) --
  node[right=15pt] {$n$}
  (1mat-4-4.south east);
\node[left=15pt] at (1mat-3-1.north west) {$U = I_{2n(2n+1)}, \qquad V = $};
\node[right=25pt] at (1mat-3-4.north east) {.};
\end{tikzpicture}
\end{center}

Conversely, if there exist $U,V \in G_{2n(2n+1)}$ such that $UXV = Y$, then we consider what structure the matrices $U$ and $V$ could possibly have by considering them as products of call matrices $C[i,j]$.

Recall that left multiplication of $X$ by $C[i,j]$ replaces rows $i$ and $j$ of $X$ with their (element-wise) maximum, whilst right multiplication by $C[i,j]$ replaces columns $i$ and $j$ of $X$ with their maximum. Thus, by regarding $U$ and $V$ as products of call generators, we see that $Y$ can be built from $X$ by successively replacing either two rows or two columns with their maximum. The last $(2n)^2$ rows of $X$ all contain $1$s in each of the last $2n$ columns, whilst the first $2n$ rows of $Y$ each have a $0$ in one of the last $2n$ columns. It follows that $U$ cannot contain $C[i,j]$ as a factor if $i \leq 2n$ and $j > 2n$. Each of the rows $n+1, \ldots, 2n$ of $X$ contain $1$s in each of the last $n$ columns, whilst the first $n$ rows of $Y$ each have a $0$ in one of the last $n$ columns. Thus $U$ cannot contain $C[i,j]$ as a factor if $i \leq n$ and $n < j \leq 2n$. The factors of $U$ are therefore call matrices $C[i,j]$ such that $i$ and $j$ are either both less than $n$, both between $n$ and $2n$, or both greater than $2n$. Both $X$ and $Y$ contain the $n \times n$ identity matrix in the top right corner. It follows that $U$ cannot contain $C[i,j]$ as a factor if $i,j \leq n$ are distinct, as this would result in an off-diagonal zero in this $n \times n$ submatrix of $Y$. Therefore, for some $D \in G_n, E \in G_{(2n)^2}$ we must have

\begin{center}
\begin{tikzpicture}[
style1/.style={
  matrix of math nodes,
  every node/.append style={text width=#1,align=center,minimum height=5ex},
  nodes in empty cells,
  left delimiter=[,
  right delimiter=],
  },
style2/.style={
  matrix of math nodes,
  every node/.append style={text width=#1,align=center,minimum height=5ex},
  nodes in empty cells,
  left delimiter=\lbrace,
  right delimiter=\rbrace,
  }
]
\matrix[style1=0.65cm] (1mat)
{
  & & & \\
  & & & \\
  & & & \\
  & & & \\
};
\draw[dashed]
  (1mat-2-1.south west) -- (1mat-2-4.south east);
\draw[dashed]
  (1mat-1-2.north east) -- (1mat-4-2.south east);
\draw[dashed]
  (1mat-1-1.south west) -- (1mat-1-2.south east);
\draw[dashed]
  (1mat-1-1.north east) -- (1mat-2-1.south east);
\node[font=\Large]
  at (1mat-1-1) {$I_n$};
\node[font=\Large]
  at (1mat-1-2) {$0$};
\node[font=\Large]
  at (1mat-2-1) {$0$};
\node[font=\Large]
  at (1mat-2-2) {$D$};
\node[font=\Large]
  at (1mat-1-3.south east) {$0$};
\node[font=\Large]
  at (1mat-3-1.south east) {$0$};
\node[font=\Large]
  at (1mat-3-3.south east) {$E$};
\draw[decoration={brace,raise=7pt},decorate]
  (1mat-1-1.north west) --
  node[above=8pt] {$n$}
  (1mat-1-1.north east);
\draw[decoration={brace,raise=7pt},decorate]
  (1mat-1-2.north west) --
  node[above=8pt] {$n$}
  (1mat-1-2.north east);
\draw[decoration={brace,raise=7pt},decorate]
  (1mat-1-3.north west) --
  node[above=8pt] {$(2n)^2$}
  (1mat-1-4.north east);
\draw[decoration={brace,raise=12pt},decorate]
  (1mat-1-4.north east) --
  node[right=15pt] {$n$}
  (1mat-1-4.south east);
\draw[decoration={brace,raise=12pt},decorate]
  (1mat-2-4.north east) --
  node[right=15pt] {$n$}
  (1mat-2-4.south east);
\draw[decoration={brace,raise=12pt},decorate]
  (1mat-3-4.north east) --
  node[right=15pt] {$(2n)^2$}
  (1mat-4-4.south east);
\node[left=15pt] at (1mat-3-1.north west) {$U = $};
\node[right=25pt] at (1mat-3-4.north east) {.};
\end{tikzpicture}
\end{center}
It follows that $X$ and $UX$ are identical in the first $n$ rows. But we have $X \preceq UX \preceq UXY = Y$ and $X$ and $Y$ are identical in all the other rows, so it must be
that $X = UX$, and $XV = UXV = Y$.

Consider now the action of $V$ by right multiplication on $X$. The first $(2n)^2$ columns of $X$ all contain $1$s in each of the first $2n$ rows, whilst the last $2n$ columns of $Y$ each have a $0$ in one of the first $2n$ rows, and so $V$ cannot contain $C[i,j]$ as a factor if $i \leq (2n)^2$ and $j > (2n)^2$. Each of the last $n$ columns of $X$ contains a $1$ in row $n+1$, whilst columns $(2n)^2+1, \ldots, (2n)^2+n$ of $Y$ each contain a $0$ in row $n+1$, and so $V$ cannot contain $C[i,j]$ as a factor if $(2n)^2 < i \leq (2n)^2 + n$ and $j > (2n)^2 + n$. It now follows that there must be some $F \in G_{(2n)^2}, G, H \in G_n$ such that

\begin{center}
\begin{tikzpicture}[
style1/.style={
  matrix of math nodes,
  every node/.append style={text width=#1,align=center,minimum height=5ex},
  nodes in empty cells,
  left delimiter=[,
  right delimiter=],
  },
style2/.style={
  matrix of math nodes,
  every node/.append style={text width=#1,align=center,minimum height=5ex},
  nodes in empty cells,
  left delimiter=\lbrace,
  right delimiter=\rbrace,
  }
]
\matrix[style1=0.65cm, right = 4.5cm] (2mat)
{
  & & & \\
  & & & \\
  & & & \\
  & & & \\
};
\draw[dashed]
  (2mat-2-1.south west) -- (2mat-2-4.south east);
\draw[dashed]
  (2mat-1-2.north east) -- (2mat-4-2.south east);
\draw[dashed]
  (2mat-3-3.south west) -- (2mat-3-4.south east);
\draw[dashed]
  (2mat-3-3.north east) -- (2mat-4-3.south east);
\node[font=\Large]
  at (2mat-1-1.south east) {$F$};
\node[font=\Large]
  at (2mat-1-3.south east) {$0$};
\node[font=\Large]
  at (2mat-3-1.south east) {$0$};
\node[font=\Large]
  at (2mat-3-3) {$G$};
\node[font=\Large]
  at (2mat-3-4) {$0$};
\node[font=\Large]
  at (2mat-4-3) {$0$};
\node[font=\Large]
  at (2mat-4-4) {$H$};
\draw[decoration={brace,raise=7pt},decorate]
  (2mat-1-1.north west) --
  node[above=8pt] {$(2n)^2$}
  (2mat-1-2.north east);
\draw[decoration={brace,raise=7pt},decorate]
  (2mat-1-3.north west) --
  node[above=8pt] {$n$}
  (2mat-1-3.north east);
\draw[decoration={brace,raise=7pt},decorate]
  (2mat-1-4.north west) --
  node[above=8pt] {$n$}
  (2mat-1-4.north east);
\draw[decoration={brace,raise=12pt},decorate]
  (2mat-1-4.north east) --
  node[right=15pt] {$(2n)^2$}
  (2mat-2-4.south east);
\draw[decoration={brace,raise=12pt},decorate]
  (2mat-3-4.north east) --
  node[right=15pt] {$n$}
  (2mat-3-4.south east);
\draw[decoration={brace,raise=12pt},decorate]
  (2mat-4-4.north east) --
  node[right=15pt] {$n$}
  (2mat-4-4.south east);
\node[left=15pt] at (2mat-3-1.north west) {$V = $};
\node[right=25pt] at (2mat-3-4.north east) {,};
\end{tikzpicture}
\end{center}

and then $XV=Y$ tells us that $AG=B$ as required.\end{proof}

\section{NP-Completeness of the Gossip Membership Problem}

In this section we prove that the Gossip Membership Problem is NP-complete.

\begin{theorem}
\label{mainthm}
GMP is NP-complete.
\end{theorem}

\begin{proof}
We have already seen that GMP is in NP. We show NP-hardness with a polynomial time reduction from MGTP. Let $A, B \in \mathbb{B}_n$ with $A$ satisfying the maximal column condition. Specifically, we show that for each instance $(A,B)$ of the Maximal Gossip Transformation Problem, there exists a Boolean matrix $C \in B_{n(n+4)}$, which can be constructed in polynomial time, for which the problem of deciding membership in the gossip monoid is equivalent to the decision problem MGTP for the pair $(A,B)$. If $n = 1$ then we simply set $C = [1]$ if $A = B$ and $C = [0]$ if $A \neq B$.  Suppose, then, that $n \geq 2$. Let $C$ be the $n(n + 4) \times n(n + 4)$ matrix

\begin{center}
\begin{tikzpicture}[
style1/.style={
  matrix of math nodes,
  every node/.append style={text width=#1,align=center,minimum height=5ex},
  nodes in empty cells,
  left delimiter=[,
  right delimiter=],
  },
style2/.style={
  matrix of math nodes,
  every node/.append style={text width=#1,align=center,minimum height=5ex},
  nodes in empty cells,
  left delimiter=\lbrace,
  right delimiter=\rbrace,
  }
]
\matrix[style1=0.65cm] (1mat)
{
  & & & & & \\
  & & & & & \\
  & & & & & \\
  & & & & & \\
  & & & & & \\
  & & & & & \\
};
\draw[dashed]
  (1mat-1-1.south west) -- (1mat-1-6.south east);
\draw[dashed]
  (1mat-3-1.south west) -- (1mat-3-6.south east);
\draw[dashed]
  (1mat-4-1.south west) -- (1mat-4-6.south east);
\draw[dashed]
  (1mat-5-1.south west) -- (1mat-5-6.south east);
\draw[dashed]
  (1mat-1-2.north east) -- (1mat-6-2.south east);
\draw[dashed]
  (1mat-1-3.north east) -- (1mat-6-3.south east);
\draw[dashed]
  (1mat-1-4.north east) -- (1mat-6-4.south east);
\draw[dashed]
  (1mat-1-5.north east) -- (1mat-6-5.south east);
\node[font=\Large]
  at (1mat-1-2.west) {$1$};
\node[font=\Large]
  at (1mat-1-3) {$A$};
\node[font=\Large]
  at (1mat-1-4) {$A$};
\node[font=\Large]
  at (1mat-1-5) {$0$};
\node[font=\Large]
  at (1mat-1-6) {$B$};
\node[font=\Large]
  at (1mat-3-2.north west) {$1$};
\node[font=\Large]
  at (1mat-3-3.north) {$1$};
\node[font=\Large]
  at (1mat-3-4.north) {$1$};
\node[font=\Large]
  at (1mat-3-5.north) {$0$};
\node[font=\Large]
  at (1mat-3-6.north) {$1$};
\node[font=\Large]
  at (1mat-4-2.west) {$0$};
\node[font=\Large]
  at (1mat-4-3) {$0$};
\node[font=\Large]
  at (1mat-4-4) {$1$};
\node[font=\Large]
  at (1mat-4-5) {$I_n$};
\node[font=\Large]
  at (1mat-4-6) {$1$};
\node[font=\Large]
  at (1mat-5-2.west) {$0$};
\node[font=\Large]
  at (1mat-5-3) {$I_n$};
\node[font=\Large]
  at (1mat-5-4) {$1$};
\node[font=\Large]
  at (1mat-5-5) {$I_n$};
\node[font=\Large]
  at (1mat-5-6) {$1$};
\node[font=\Large]
  at (1mat-6-2.west) {$0$};
\node[font=\Large]
  at (1mat-6-3) {$I_n$};
\node[font=\Large]
  at (1mat-6-4) {$1$};
\node[font=\Large]
  at (1mat-6-5) {$I_n$};
\node[font=\Large]
  at (1mat-6-6) {$1$};
\draw[decoration={brace,raise=7pt},decorate]
  (1mat-1-1.north west) --
  node[above=8pt] {$n^{2}$}
  (1mat-1-2.north east);
\draw[decoration={brace,raise=7pt},decorate]
  (1mat-1-3.north west) --
  node[above=8pt] {$n$}
  (1mat-1-3.north east);
\draw[decoration={brace,raise=7pt},decorate]
  (1mat-1-4.north west) --
  node[above=8pt] {$n$}
  (1mat-1-4.north east);
\draw[decoration={brace,raise=7pt},decorate]
  (1mat-1-5.north west) --
  node[above=8pt] {$n$}
  (1mat-1-5.north east);
\draw[decoration={brace,raise=7pt},decorate]
  (1mat-1-6.north west) --
  node[above=8pt] {$n$}
  (1mat-1-6.north east);
\draw[decoration={brace,raise=12pt},decorate]
  (1mat-1-6.north east) --
  node[right=15pt] {$n$}
  (1mat-1-6.south east);
\draw[decoration={brace,raise=12pt},decorate]
  (1mat-2-6.north east) --
  node[right=15pt] {$n^{2}$}
  (1mat-3-6.south east);
\draw[decoration={brace,raise=12pt},decorate]
  (1mat-4-6.north east) --
  node[right=15pt] {$n$}
  (1mat-4-6.south east);
\draw[decoration={brace,raise=12pt},decorate]
  (1mat-5-6.north east) --
  node[right=15pt] {$n$}
  (1mat-5-6.south east);
\draw[decoration={brace,raise=12pt},decorate]
  (1mat-6-6.north east) --
  node[right=15pt] {$n$}
  (1mat-6-6.south east);
\node[left=15pt] at (1mat-3-1.south west) {$C = $};
\node[right=30pt] at (1mat-3-6.south east) {.};
\end{tikzpicture}
\end{center}
This matrix can clearly be constructed in polynomial time. We claim that $C \in G_{n(n+4)}$ if and only if there is a $G \in G_n$ such that $AG = B$.

For ease of reference during the proof, we shall label the blocks of the matrix (and other matrices of the same size) as follows:

\begin{center}
\begin{tikzpicture}[
style1/.style={
  matrix of math nodes,
  every node/.append style={text width=#1,align=center,minimum height=5ex},
  nodes in empty cells,
  left delimiter=[,
  right delimiter=],
  },
style2/.style={
  matrix of math nodes,
  every node/.append style={text width=#1,align=center,minimum height=5ex},
  nodes in empty cells,
  left delimiter=\lbrace,
  right delimiter=\rbrace,
  }
]
\matrix[style1=0.65cm] (1mat)
{
  & & & & & \\
  & & & & & \\
  & & & & & \\
  & & & & & \\
  & & & & & \\
  & & & & & \\
};
\draw[dashed]
  (1mat-1-1.south west) -- (1mat-1-6.south east);
\draw[dashed]
  (1mat-3-1.south west) -- (1mat-3-6.south east);
\draw[dashed]
  (1mat-4-1.south west) -- (1mat-4-6.south east);
\draw[dashed]
  (1mat-5-1.south west) -- (1mat-5-6.south east);
\draw[dashed]
  (1mat-1-2.north east) -- (1mat-6-2.south east);
\draw[dashed]
  (1mat-1-3.north east) -- (1mat-6-3.south east);
\draw[dashed]
  (1mat-1-4.north east) -- (1mat-6-4.south east);
\draw[dashed]
  (1mat-1-5.north east) -- (1mat-6-5.south east);
\node[font=\Large]
  at (1mat-1-2.west) {$a1$};
\node[font=\Large]
  at (1mat-1-3) {$b1$};
\node[font=\Large]
  at (1mat-1-4) {$c1$};
\node[font=\Large]
  at (1mat-1-5) {$d1$};
\node[font=\Large]
  at (1mat-1-6) {$e1$};
\node[font=\Large]
  at (1mat-3-2.north west) {$a2$};
\node[font=\Large]
  at (1mat-3-3.north) {$b2$};
\node[font=\Large]
  at (1mat-3-4.north) {$c2$};
\node[font=\Large]
  at (1mat-3-5.north) {$d2$};
\node[font=\Large]
  at (1mat-3-6.north) {$e2$};
\node[font=\Large]
  at (1mat-4-2.west) {$a3$};
\node[font=\Large]
  at (1mat-4-3) {$b3$};
\node[font=\Large]
  at (1mat-4-4) {$c3$};
\node[font=\Large]
  at (1mat-4-5) {$d3$};
\node[font=\Large]
  at (1mat-4-6) {$e3$};
\node[font=\Large]
  at (1mat-5-2.west) {$a4$};
\node[font=\Large]
  at (1mat-5-3) {$b4$};
\node[font=\Large]
  at (1mat-5-4) {$c4$};
\node[font=\Large]
  at (1mat-5-5) {$d4$};
\node[font=\Large]
  at (1mat-5-6) {$e4$};
\node[font=\Large]
  at (1mat-6-2.west) {$a5$};
\node[font=\Large]
  at (1mat-6-3) {$b5$};
\node[font=\Large]
  at (1mat-6-4) {$c5$};
\node[font=\Large]
  at (1mat-6-5) {$d5$};
\node[font=\Large]
  at (1mat-6-6) {$e5$};
\draw[decoration={brace,raise=7pt},decorate]
  (1mat-1-1.north west) --
  node[above=8pt] {$n^2$}
  (1mat-1-2.north east);
\draw[decoration={brace,raise=7pt},decorate]
  (1mat-1-3.north west) --
  node[above=8pt] {$n$}
  (1mat-1-3.north east);
\draw[decoration={brace,raise=7pt},decorate]
  (1mat-1-4.north west) --
  node[above=8pt] {$n$}
  (1mat-1-4.north east);
\draw[decoration={brace,raise=7pt},decorate]
  (1mat-1-5.north west) --
  node[above=8pt] {$n$}
  (1mat-1-5.north east);
\draw[decoration={brace,raise=7pt},decorate]
  (1mat-1-6.north west) --
  node[above=8pt] {$n$}
  (1mat-1-6.north east);
  \draw[decoration={brace,raise=12pt},decorate]
  (1mat-1-6.north east) --
  node[right=15pt] {$n$}
  (1mat-1-6.south east);
\draw[decoration={brace,raise=12pt},decorate]
  (1mat-2-6.north east) --
  node[right=15pt] {$n^{2}$}
  (1mat-3-6.south east);
\draw[decoration={brace,raise=12pt},decorate]
  (1mat-4-6.north east) --
  node[right=15pt] {$n$}
  (1mat-4-6.south east);
\draw[decoration={brace,raise=12pt},decorate]
  (1mat-5-6.north east) --
  node[right=15pt] {$n$}
  (1mat-5-6.south east);
\draw[decoration={brace,raise=12pt},decorate]
  (1mat-6-6.north east) --
  node[right=15pt] {$n$}
  (1mat-6-6.south east);
\end{tikzpicture}
\end{center}
We shall also use $a$, $b$, $c$, $d$ and $e$ to refer to the sets of indices $\{1, \ldots, n^2\}$, $\{n^2 + 1, \ldots, n^2 + n\}$, $\{n^2 + n + 1, \ldots, n^2 + 2n\}$, $\{n^2 + 2n + 1, \ldots, n^2 + 3n\}$ and $\{n^2 + 3n + 1, \ldots, n^2 + 4n\}$ respectively, so that the columns indexed by these sets correspond to the blocks described above. We define $a_i = i, b_i = n^2 + i, c_i = n^2 + n + i, d_i = n^2 + 2n + i$ and $e_i = n^2 + 3n + i$ so that, for example, the $i$th column in block $c$ is column $c_i$.

We first show that if $AG = B$ for some $G \in G_n$ then $C \in G_{n(n+4)}$, by showing how to write $C$ as a product of call matrices. Let

\begin{center}
\begin{tikzpicture}[
style1/.style={
  matrix of math nodes,
  every node/.append style={text width=#1,align=center,minimum height=5ex},
  nodes in empty cells,
  left delimiter=[,
  right delimiter=],
  },
style2/.style={
  matrix of math nodes,
  every node/.append style={text width=#1,align=center,minimum height=5ex},
  nodes in empty cells,
  left delimiter=\lbrace,
  right delimiter=\rbrace,
  }
]
\matrix[style1=0.65cm] (1mat)
{
  & & & & & \\
  & & & & & \\
  & & & & & \\
  & & & & & \\
  & & & & & \\
  & & & & & \\
};
\draw[dashed]
  (1mat-1-1.south west) -- (1mat-1-6.south east);
\draw[dashed]
  (1mat-3-1.south west) -- (1mat-3-6.south east);
\draw[dashed]
  (1mat-4-1.south west) -- (1mat-4-6.south east);
\draw[dashed]
  (1mat-5-1.south west) -- (1mat-5-6.south east);
\draw[dashed]
  (1mat-1-2.north east) -- (1mat-6-2.south east);
\draw[dashed]
  (1mat-1-3.north east) -- (1mat-6-3.south east);
\draw[dashed]
  (1mat-1-4.north east) -- (1mat-6-4.south east);
\draw[dashed]
  (1mat-1-5.north east) -- (1mat-6-5.south east);
\node[font=\Large]
  at (1mat-1-2.west) {$1$};
\node[font=\Large]
  at (1mat-1-3) {$A$};
\node[font=\Large]
  at (1mat-1-4) {$0$};
\node[font=\Large]
  at (1mat-1-5) {$0$};
\node[font=\Large]
  at (1mat-1-6) {$0$};
\node[font=\Large]
  at (1mat-3-2.north west) {$1$};
\node[font=\Large]
  at (1mat-3-3.north) {$1$};
\node[font=\Large]
  at (1mat-3-4.north) {$0$};
\node[font=\Large]
  at (1mat-3-5.north) {$0$};
\node[font=\Large]
  at (1mat-3-6.north) {$0$};
\node[font=\Large]
  at (1mat-4-2.west) {$0$};
\node[font=\Large]
  at (1mat-4-3) {$0$};
\node[font=\Large]
  at (1mat-4-4) {$I_n$};
\node[font=\Large]
  at (1mat-4-5) {$0$};
\node[font=\Large]
  at (1mat-4-6) {$0$};
\node[font=\Large]
  at (1mat-5-2.west) {$0$};
\node[font=\Large]
  at (1mat-5-3) {$0$};
\node[font=\Large]
  at (1mat-5-4) {$0$};
\node[font=\Large]
  at (1mat-5-5) {$I_n$};
\node[font=\Large]
  at (1mat-5-6) {$0$};
\node[font=\Large]
  at (1mat-6-2.west) {$0$};
\node[font=\Large]
  at (1mat-6-3) {$0$};
\node[font=\Large]
  at (1mat-6-4) {$0$};
\node[font=\Large]
  at (1mat-6-5) {$0$};
\node[font=\Large]
  at (1mat-6-6) {$I_n$};
\draw[decoration={brace,raise=7pt},decorate]
  (1mat-1-1.north west) --
  node[above=8pt] {$a$}
  (1mat-1-2.north east);
\draw[decoration={brace,raise=7pt},decorate]
  (1mat-1-3.north west) --
  node[above=8pt] {$b$}
  (1mat-1-3.north east);
\draw[decoration={brace,raise=7pt},decorate]
  (1mat-1-4.north west) --
  node[above=8pt] {$c$}
  (1mat-1-4.north east);
\draw[decoration={brace,raise=7pt},decorate]
  (1mat-1-5.north west) --
  node[above=8pt] {$d$}
  (1mat-1-5.north east);
\draw[decoration={brace,raise=7pt},decorate]
  (1mat-1-6.north west) --
  node[above=8pt] {$e$}
  (1mat-1-6.north east);
\node[left=15pt] at (1mat-3-1.south west) {$Y_1 = $};
\node[right=10pt] at (1mat-3-6.south east) {,};

\matrix[style1=0.65cm, right=5.25cm] (2mat)
{
  & & & \\
  & & & \\
  & & & \\
  & & & \\
};
\draw[dashed]
  (2mat-3-1.south west) -- (2mat-3-4.south east);
\draw[dashed]
  (2mat-1-3.north east) -- (2mat-4-3.south east);
\node[font=\Large]
  at (2mat-2-2) {$I_{n(n+3)}$};
\node[font=\Large]
  at (2mat-2-4) {$0$};
\node[font=\Large]
  at (2mat-4-2) {$0$};
\node[font=\Large]
  at (2mat-4-4) {$G$};
\draw[decoration={brace,raise=7pt},decorate]
  (2mat-1-1.north west) --
  node[above=8pt] {$a,b,c,d$}
  (2mat-1-3.north east);
\draw[decoration={brace,raise=7pt},decorate]
  (2mat-1-4.north west) --
  node[above=8pt] {$e$}
  (2mat-1-4.north east);
\node[left=15pt] at (2mat-2-1.south west) {$Y_7 = $};
\node[right=10pt] at (2mat-2-4.south east) {,};
\end{tikzpicture}
\end{center}

\begin{align*}
& Y_2 = \prod_{i = 1}^n C[d_i, e_i],
& Y_3 = \prod_{i = 1}^n C[c_i, d_i], \qquad\qquad
& Y_4 = C[\{c_1, \ldots, c_n\}], \\
& Y_5 = \prod_{i = 1}^n C[b_i, e_i],
& Y_6 = \prod_{i = 1}^n C[c_i, e_i]. \qquad\qquad
&
\end{align*}

Since $n \geq 2$ it follows from Lemma \ref{boolingoss} that $Y_1$ is a product of call matrices in $G_{n(n+4)}$. The assumption that $G \in G_n$ means that $Y_7$ can be written as a product of call matrices in $G_{n(n+4)}$. Matrices $Y_2$ through to $Y_6$ are all explicitly defined as products of call matrices, and we note that the products in $Y_2, Y_3, Y_5$ and $Y_6$ can be taken in any order. It is then straightforward to check that $C = Y_1 Y_2 Y_3 Y_4 Y_5 Y_6 Y_7$.

Conversely, assume that $C \in G_{n(n+4)}$ and fix a sequence of call matrices $C_1, \dots, C_q$ such that $C = C_1 \cdots C_q$. We may clearly assume without loss of generality that the sequence contains no redundant factors, in the sense that $C_1\cdots C_t \neq C_1 \cdots C_t C_{t+1}$ (which since multiplication is monotonic means
that $C_1\cdots C_t \prec C_1 \cdots C_t C_{t+1}$) for all $t$. By defining a particular scattered subsequence of these factors, we shall construct an element
$G \in G_n$ satisfying $AG=B$.

We start by observing that for any scattered subsequence $C_{t_1}, \ldots, C_{t_p}$ of $C_1, \ldots, C_q$ we must have $C_{t_1}\cdots C_{t_p} \preceq C$. We proceed by establishing a series of claims.
\medskip

\begin{claim}
For each $t$, if $C_t$ is of the form $C[i,j]$ with $i \in a \cup b$ and $j \in c \cup d \cup e$, then $C_t = C[b_k, e_k]$ for some $k \in \lbrace 1,\ldots, n \rbrace$. Moreover, for each $k$, the matrix $C[b_k, e_k]$ must appear exactly once in the sequence $C_1, \dots, C_q$.
\end{claim}

\medskip

\begin{claimproof}
Suppose that $C_t = C[i,j]$, where $i \in a \cup b$ and $j \in c \cup d \cup e$. Since $C_t \preceq C$ we note that $i \notin a$, otherwise $C$ would contain a $1$ in  block $a3$, $a4$ or $a5$. Similarly, we note that $j \notin c \cup d$, since otherwise $C$ would contain a 1 in block $b3$ or $d2$. Finally, let $i=b_k$ and note that if $j \neq e_k$ then $C$ would contain an off-diagonal 1 in block $b5$.

It now follows that the only way that $C$ can contain the identity matrix in block $b5$ is for each of the matrices $C[b_k, e_k]$ to occur at least once in the sequence $C_1, \dots, C_q$. Assume for contradiction that $C[b_k, e_k]$ occurs more than once in the sequence, and suppose that $C_r$ and $C_s$ are the first two occurrences of this matrix. If $C_t$ is a factor occurring between
$C_r$ and $C_s$ then $C_rC_tC_s \preceq C$. We cannot have $C_t$ of the form $C[b_k, j]$ with $j \in a \cup b$, as otherwise $C_rC_tC_s$ --- and therefore also $C$ --- would have a $1$ in block $a5$ or an off-diagonal $1$ in block $b5$. Similarly, it cannot be of the form $C[e_k, j]$ with $j \in c \cup d \cup e$, since then $C_rC_tC_s$ would have a $1$ in block $b3$ or $d2$ or an off-diagonal $1$ in block $b5$. By the first part of the claim, $C_t$ cannot be of the form $C[b_k, j]$ with $j \in c \cup d \cup e \setminus\{e_k\}$ or $C[e_k, j]$ with $j \in a \cup b \setminus\{b_k\}$. It also cannot be of the form $C[b_k, e_k]$ as it lies strictly between the first two factors of this form. Therefore $C_t$ is not of the form $C[b_k, j]$ or $C[e_k, j]$ for any $j$, so $C_r$ commutes with $C_t$. Since $C_t$ was an arbitrary factor between $C_r$ and $C_s$, and since $C_r = C_s$ is an idempotent, we have $C_r C_{r+1} \cdots C_{s-1} C_s = C_r C_s C_{r+1} \cdots C_{s-1} = C_r C_{r+1} \cdots C_{s-1}$, contradicting our assumption that the sequence $C_1, \dots, C_q$ contains no redundant factors and thus proving the claim.\end{claimproof}
\medskip

For $k \in \{1, \ldots n\}$ let $w_k$ be the unique index such that $C_{w_k} = C[b_k, e_k]$.
\medskip

\begin{claim}
For each $k \leq n$, column $k$ of block $e1$ in the product $C_1 \cdots C_{w_k -1}$ is equal to the zero vector, but the same column in $C_1 \cdots C_{w_k}$ is equal to column $k$ of $A$.
\end{claim}

\medskip

\begin{claimproof}
We first show that in the product $C_1 \dots C_{w_k-1}$, the $k$th column of block $e1$ is equal to the zero vector. Assume that the vector is non-zero, so there is some $C_s$ with $s < w_k$ such that $C_1 \cdots C_{s-1}$ has the zero vector in this location, but $C_1 \cdots C_s$ does not. The factor $C_s$ must be of the form $C[e_k, j]$ and by Claim 1 and the fact that $s < w_k$, we have $j \in c \cup d \cup e$. We note that $j \notin c$, otherwise $C_s C_{w_k}$ --- and hence also $C$ --- would have a $1$ in block $b3$. Similarly we note that $j \notin e$, otherwise $C_s C_{w_k}$ would have an off-diagonal $1$ in block $b5$. This means $C_s$ must be of the form $C[e_k, d_i]$, but then columns $e_k$ and $d_i$ of $C_1 \cdots C_s$ are identical to each other, so this product has a $1$ in block $d1$, contradicting $C_1 \cdots C_s \preceq C$. Thus the $k$th column of block $e1$ in the product $C_1 \cdots C_{w_k-1}$ is equal to the zero vector.

Since the last element of the product $C_1 \cdots C_{w_k}$ is $C_{w_k} = C[b_k, e_k]$, columns $b_k$ and $e_k$ of the product are identical and the second part of the claim will follow if the $k$th column of block $b1$ in this product is equal to the $k$th column of $A$. Suppose this does not hold; since it does hold in $C$, the sequence $C_{w_k +1}, \dots, C_q$ must include a subsequence which transforms the $k$th column of $b1$ into the $k$th column of $A$. Let $C[b_k, j]$ be one such call. Notice that $j \notin a$, since otherwise $C_{w_k}C[b_k, j] \preceq C$ would have a 1 in block $a5$. Similarly, $j \notin b$, since otherwise $C_{w_k}C[b_k, j]$ would have an off-diagonal 1 in block $b5$. The first part of Claim 1 now shows that the only remaining possibility is $j=e_k$. However, this would result in two factors of $C$ both equal to $C[b_k, e_k]$, contradicting the second part of Claim 1. Hence column $k$ of block $b1$ in the product $C_1 \cdots C_{w_k}$ is equal to the $k$th column of $A$, completing the proof of the claim.
\end{claimproof}
\medskip

To provide a convenient starting point for induction, let $C_0$ be the $(n^2+4n)\times (n^2+4n)$ identity matrix, so that $C= C_0C_1 \cdots C_q$.
\medskip

\begin{claim}
For all $0\leq t\leq q$, every column in block $c1$, $d1$ or $e1$ of $C_0C_1 \cdots C_t$ is equal to either $\underline{0}$ or the maximum of some collection of columns of $A$.
\end{claim}

\medskip

\begin{claimproof}
We prove this by induction on $t$. When $t = 0$ this is clearly true, as each of these columns is equal to $\underline{0}$. For $t \geq 1$ we shall assume that the condition holds for the columns of the appropriate blocks in $C_1 \cdots C_{t-1}$. By Claim 1, $C_t$ is of the form (i) $C[i,j]$ with $i,j \in a \cup b$, (ii) $C[i,j]$ with $i,j \in c \cup d \cup e$ or (iii) $C[b_k, e_k]$ for some $k$. In case (i) the columns of blocks $c1$, $d1$ and $e1$ are the same as the corresponding columns in the product $C_1 \cdots C_{t-1}$, which satisfy the condition by assumption. In case (ii) two of the columns of blocks $c1$, $d1$ and $e1$ are equal to the maximum of the corresponding columns from the product $C_1 \cdots C_{t-1}$, and so they satisfy the condition of the claim, and the rest of the columns are the same as the corresponding columns from $C_1 \cdots C_{t-1}$. In case (iii) by Claim 2, column $k$ of block $e1$ is equal to the $k$th column of $A$, so satisfies the claim, and all other columns of blocks $c1$, $d1$ and $e1$ are equal to the corresponding columns in $C_1 \cdots C_{t-1}$. The claim follows by induction. \end{claimproof}
\medskip

To prove the theorem we want to construct a matrix $G \in G_n$ such that $AG = B$. Consider those call matrices $C_t$ (in order) for which there exists $k \leq n$ such that $w_k<t$ and column $k$ of block $e1$ in the product $C_1 \cdots C_t$ is strictly larger than the same column in the product $C_1 \cdots C_{t-1}$. (Intuitively, these are the factors which modify some column $k$ of block $e1$ subsequent to the factor $C_{w_k} = C[b_k,e_k]$ which copies column $k$ of $A$ into this column.) Let $C_{t_1}, \ldots, C_{t_p}$ denote the subsequence of these matrices. We shall assume that this sequence is non-empty, since otherwise $B = A$ and $AG = B$ is trivially satisfied by $G = I_n$.

\medskip

\begin{claim}
Each element of the sequence $C_{t_1}, \ldots, C_{t_p}$ is of the form $C_{t_r} = C[e_j, e_k]$ for some $j,k \leq n$ such that $w_j < t_r$ and $w_k < t_r$.
\end{claim}

\medskip

\begin{claimproof}
Let $C_{t_r}$ be an element of this subsequence. From the definition of the subsequence there is some $k \leq n$ such that $w_k < t_r$ and column $k$ of block $e1$ of $C_1 \cdots C_{t_r}$ is strictly larger than the same column in $C_1 \cdots C_{{t_r}-1}$. Thus $C_{t_r}$ must be of the form $C[i, e_k]$ for some $i \leq n^2 + 4n$. Claim 1 tells us that $i \notin a$. Since $t_r > w_k$, Claim 1 also tells us that $i \notin b$. By monotonicity, all columns in block $d1$ of $C_1 \cdots C_{t_r}$ are $\underline{0}$, so $i \notin d$. Assume for contradiction that $C_{t_r} = C[c_j, e_k]$ for some $j$.

We observe several relations between columns of various matrices:

\begin{itemize}
\item Column $j$ of $A$ is equal to column $j$ of block $c1$ in $C$.
\item By monotonicity, column $j$ of block $c1$ in $C$ is greater than or equal to the same column in $C_1 \cdots C_{t_r}$.
\item Since the last factor in the product $C_1 \cdots C_{t_r}$ is $C_{t_r} = C[c_j, e_k]$, columns $c_j$ and $e_k$ are equal in the resulting matrix.
\item By our choice of $k$, column $k$ of block $e1$ of $C_1 \cdots C_{t_r}$ is strictly greater than the same column in $C_1 \cdots C_{{t_r}-1}$.
\item Since $w_k < t_r$, by Claim 2 and monotonicity we know that column $k$ of block $e1$ in $C_1 \cdots C_{{t_r}-1}$ is greater than or equal to column $k$ of $A$.
\end{itemize}

Putting these together, we find that column $j$ of $A$ is strictly greater than column $k$ of $A$, which contradicts the maximal column condition. Thus $C_{t_r}$ must be of the form $C[e_j, e_k]$ for some $j$.

We know that $w_k < t_r$ by definition of the sequence, so for the second part of the claim it is enough to observe that $t_r \neq w_j$ since $C_{t_r} \neq C[b_j, e_k]$, and if $t_r < w_j$ then we would have $C[e_j, e_k]C[b_j, e_j] \preceq C$ and so $C$ would contain an off-diagonal $1$ in block $b5$, which is not the case.
\end{claimproof}

\medskip

We see from this claim that there are only two types of call matrix in the sequence which can modify a column of block $e1$. The first matrix in the sequence which modifies column $i$ of block $e1$ is $C_{w_i}$, and all subsequent matrices (if any) which modify this column are elements of the subsequence $C_{t_1}, \ldots, C_{t_p}$.

We now define a new sequence of call matrices in $G_n$. For $1 \leq r \leq p$, if $C_{t_r} = C[e_i, e_j]$ then let $D_r = C[i,j] \in G_n$.

\medskip
\begin{claim}
For each $r \leq p$ and $k \leq n$, column $k$ of block $e1$ of $C_1 \cdots C_{t_r}$ is equal to the zero vector if $t_r < w_k$, or equal to column $k$ of $A D_1 \cdots D_r$ otherwise.
\end{claim}

\medskip

\begin{claimproof}
By Claim 2, for each $k \leq n$ the $k$th column of block $e1$ in the product $C_1 \cdots C_{w_k-1}$ is equal to the zero vector. By monotonicity, if $t_r < w_k$ then the same column is equal to the zero vector in the product $C_1 \cdots C_{t_r}$. We prove by induction on $r$ that for each $k$ such that $w_k < t_r$, column $k$ of block $e1$ of $C_1\cdots C_{t_r}$ is equal to column $k$ of $A D_1 \cdots D_r$.

Let $t_0 = 0$. Then $t_0 < w_k$ for all $k$ so the claim holds for $r = 0$. Now assume the claim holds for $r-1$. That is, for each $k$ such that $w_k < t_{(r-1)}$, column $k$ of block $e1$ of $C_1\cdots C_{t_{(r-1)}}$ is equal to column $k$ of $A D_1 \cdots D_{r-1}$. Given $k \leq n$ such that $w_k < t_r$ we first consider the differences between column $k$ of block $e1$ in $C_1\cdots C_{t_{(r-1)}}$ and the same column in $C_1\cdots C_{t_r -1}$. No elements of the subsequence $C_{t_1}, \ldots, C_{t_p}$ occur among the matrices $C_{t_{(r-1)} + 1}, \ldots, C_{t_r -1}$ and we have noted that the only other factor of $C$ which can possibly modify column $k$ of block $e1$ is $C_{w_k}$. Thus the two columns under consideration are identical unless $t_{(r-1)} < w_k < t_r$, and in this case, by Claim 2, column $k$ of block $e1$ in $C_1\cdots C_{t_r -1}$ is equal to column $k$ of $A$. If $t_{(r-1)} < w_k < t_r$ then, since $C_{w_k}$ comes earlier in the sequence than any other matrix which can modify column $k$ of block $e1$, we see that $C_{t_r}$ is the first element of the sequence $C_{t_1}, \ldots, C_{t_p}$ which can possibly modify column $k$ of block $e1$, and thus $D_r$ is the first element of the sequence $D_1, \ldots, D_p$ which can possibly modify column $k$. Thus column $k$ of $AD_1 \cdots D_{r-1}$ is equal to column $k$ of $A$, and is therefore equal to column $k$ of $e1$ in $C_1\cdots C_{t_r -1}$. If $t_{(r-1)} < w_k < t_r$ does not hold then, since we are only considering $k$ such that $w_k < t_r$, we must have $w_k < t_{(r-1)}$. Thus by the inductive hypothesis we know that column $k$ of block $e1$ of $C_1\cdots C_{t_{(r-1)}}$ is equal to column $k$ of $A D_1 \cdots D_{r-1}$. Since none of the matrices $C_{t_{(r-1)} + 1}, \ldots, C_{t_r -1}$ modify this column, column $k$ of $e1$ in $C_1\cdots C_{t_r -1}$ is again equal to column $k$ of $A D_1 \cdots D_{r-1}$. Thus for all $k$ such that $w_k < t_r$ we know that column $k$ of $e1$ in $C_1\cdots C_{t_r -1}$ is equal to column $k$ of $A D_1 \cdots D_{r-1}$.

By Claim 4, $C_{t_r} = C[e_i, e_j]$ for some $i,j \leq n$ such that $w_i, w_j < t_r$, and then $D_r$ is defined to be $C[i,j]$. Let $k < n$ be such that $w_k < t_r$. We first consider the case when $k \notin \{i,j\}$. Since $k \notin \{i,j\}$, column $k$ of block $e1$ of $C_1\cdots C_{t_r}$ is equal to the same column in $C_1\cdots C_{t_r -1}$. By the previous paragraph, this column is equal to column $k$ of $AD_1 \cdots D_{r-1}$, and again because $k \notin \{i,j\}$, this is equal to column $k$ of $AD_1 \cdots D_r$. In the case where $k \in \{i,j\}$, column $k$ of block $e1$ of $C_1\cdots C_{t_r}$ is equal to the maximum of columns $i$ and $j$ of block $e1$ of $C_1\cdots C_{t_r -1}$. Since $w_i, w_j < t_r$, by the previous paragraph this is equal to the maximum of columns $i$ and $j$ of $AD_1 \cdots D_{r-1}$. Since $D_r = C[i,j]$, column $k$ of $AD_1 \cdots D_r$ is also equal to the maximum of columns $i$ and $j$ of $AD_1 \cdots D_{r-1}$, which completes the proof of the claim.
\end{claimproof}

\medskip

We now define $G = D_1 \cdots D_p$. For each $k \leq n$, the only factors of $C$ which can possibly modify column $k$ of block $e1$ are $C_{w_k}$ and elements of the subsequence $C_{t_1}, \ldots, C_{t_p}$. Therefore column $k$ of block $e1$ of $C$ is equal to the same column in $C_1 \cdots C_{t_p}$ if $w_k \leq t_p$, or it is equal to the same column in $C_1 \cdots C_{w_k}$ if $t_p < w_k$.

If $w_k \leq t_p$ then by Claim 5 column $k$ of $e1$ in $C_1 \cdots C_{t_p}$ is equal to column $k$ of $AG = AD_1 \cdots D_p$, and therefore column $k$ of $e1$ in $C$ is also equal to column $k$ of $AG$.

If $t_p < w_k$ then $C_{w_k}$ is the last factor of $C$ to modify column $k$ of block $e1$, so by Claim 2, column $k$ of $e1$ in $C$ is equal to column $k$ of $A$. We know from Claim 4 that if $C_{t_r} = C[e_j,e_k]$ for some $j$ and $r$ then $w_k < t_r$. However $t_r \leq w_k$ for all $r \leq p$, so no element of the sequence $C_{t_1}, \ldots, C_{t_p}$ is equal to $C[e_j,e_k]$ for any $j \leq n$. Therefore no element of the sequence $D_1, \ldots, D_p$ is equal to $C[j,k]$ for any $j \leq n$, and so column $k$ of $AG = AD_1\cdots D_p$ is also equal to column $k$ of $A$.

In both cases, column $k$ of block $e1$ of $C$ is equal to column $k$ of $AG$. Therefore $AG$ is equal to block $e1$ of $C$, which is equal to $B$, and so $G \in G_n$ satisfies $AG = B$ as required to complete the proof.
\end{proof}

\bibliographystyle{plain}

\end{document}